\documentclass[a4paper]{amsart}

\usepackage[utf8]{inputenc}
\usepackage[english]{babel}
\usepackage{hyperref}
\usepackage{amssymb}
\usepackage{mathtools}
\usepackage{enumitem}
\setenumerate[1]{label=(\alph*)}
\setenumerate[2]{label=(\roman*)}
\usepackage[all]{xy}
\usepackage{tikz}
\usetikzlibrary{shapes.geometric,matrix,arrows}
\usepackage[normalem]{ulem}

\numberwithin{equation}{section}
 
\newtheorem{mtheorem}{Theorem}
\newtheorem{theorem}{Theorem}[section]
\newtheorem{proposition}[theorem]{Proposition}
\newtheorem{lemma}[theorem]{Lemma}
\newtheorem{corollary}[theorem]{Corollary}

\theoremstyle{definition}
\newtheorem{definition}[theorem]{Definition}
\newtheorem{example}[theorem]{Example}
\newtheorem{construction}[theorem]{Construction}
\newtheorem{notation}[theorem]{Notation}

\theoremstyle{remark}
\newtheorem{remark}[theorem]{Remark}

\DeclareMathOperator{\im}{Im}
\DeclareMathOperator{\kernel}{Ker}
\DeclareMathOperator{\cok}{Cok}

\DeclareMathOperator{\Hom}{Hom}

\DeclareMathOperator{\Ext}{Ext}

\DeclareMathOperator{\holim}{holim}

\DeclareMathOperator{\Spec}{Spec}

\DeclareMathOperator{\HH}{H}

\renewcommand{\epsilon}{{\varepsilon}}

\renewcommand{\phi}{{\varphi}}

\newcommand{\define}[1]{{\emph{#1}}}

\newcommand{\id}{{\operatorname{id}}}
\newcommand{\fppf}{{\operatorname{fppf}}}
\newcommand{\opp}{{\operatorname{op}}}

\newcommand{\bZ}{{\mathbb{Z}}}
\newcommand{\bN}{{\mathbb{N}}}
\newcommand{\bP}{{\mathbb{P}}}
\newcommand{\bA}{{\mathbb{A}}}

\DeclareMathOperator{\PGL}{PGL}
\DeclareMathOperator{\GL}{GL}
\DeclareMathOperator{\BB}{B}

\newcommand{\GGm}{{\ensuremath{\mathbb{G}_{\operatorname{m}}}}}
\newcommand{\GGms}[1]{{\ensuremath{\mathbb{G}_{\operatorname{m},#1}}}}

\newcommand{\AbGerbe}{{\mathcal{A}b\mathcal{G}erbe}}
\newcommand{\AbGroup}{{\mathcal{A}b\mathcal{G}roup}}

\DeclareMathOperator{\Band}{Band}

\newcommand{\liset}{{\operatorname{lis-\acute{e}t}}}

\newcommand{\qc}{{\operatorname{qc}}}
\newcommand{\pc}{{\operatorname{pc}}}

\newcommand{\co}{{\operatorname{coh}}}
\newcommand{\pf}{{\operatorname{pf}}}
\newcommand{\sg}{{\operatorname{sg}}}
\newcommand{\bd}{{\operatorname{b}}}
\newcommand{\locbd}{{\operatorname{lb}}}
\newcommand{\D}{{\operatorname{D}}}

\newcommand{\coh}{{\ensuremath{\operatorname{Coh}}}}

\newcommand{\Rd}{\mathsf{R}}
\newcommand{\Ld}{\mathsf{L}}

\newcommand{\ra}{\rightarrow}

\newcommand{\xra}[2][]{\xrightarrow[{#1}]{#2}}

\newcommand{\sira}{\xra{\sim}}

\newcommand{\Qcoh}{{\ensuremath{\operatorname{Qcoh}}}}
\newcommand{\Mod}{{\ensuremath{\operatorname{Mod}}}}
\newcommand{\sheafHom}{{\mathcal{H}{om}}}

\newcommand{\sptag}[1]{\href{http://stacks.math.columbia.edu/tag/#1}{Tag~#1}}

\hypersetup{  
  pdftitle    = {Decompositions of derived categories of gerbes
    and families of Brauer--Severi varieties},
  pdfauthor   = {Daniel Bergh, Olaf M.~Schn{\"u}rer},
  colorlinks=true}

\title[Decompositions]{Decompositions of derived categories of
  gerbes and of families of Brauer--Severi varieties}
\begin{document}

\author{Daniel Bergh \and Olaf M.~Schn{\"u}rer}

\address{       
  Department of Mathematical Sciences\\
  Copenhagen University\\
  Universitetsparken~5\\
  2100~København~Ø\\ 
  Denmark
}
\email{dbergh@math.ku.dk}

\address{       
  Institut f\"ur Mathematik\\
  Universit\"at Paderborn\\
  Warburger Str. 100\\
  33098 Paderborn\\
  Germany
}

\email{olaf.schnuerer@math.uni-paderborn.de}


\begin{abstract}
It is well known that the category of quasi-coherent sheaves on a
gerbe banded by a diagonalizable group decomposes according to
the characters of the group. 
We establish the corresponding decomposition of the unbounded
derived category of complexes of sheaves with quasi-coherent
cohomology. This generalizes earlier work by Lieblich for gerbes
over schemes whereas our gerbes may live over arbitrary algebraic stacks. 

By combining this decomposition with the semi-orthogonal
decomposition for a 
projectivized vector bundle, we deduce a semi-orthogonal
decomposition of the derived category of a family of
Brauer--Severi varieties whose components can be described in terms
of twisted sheaves on the base. This reproves and generalizes a
result of Bernardara.
\end{abstract}


\maketitle
\tableofcontents

\section{Introduction}

The concept of a \emph{twisted sheaf} goes
back to Giraud 
and his formal treatment of the Brauer group in terms of \emph{gerbes}~\cite[Section~V.4]{giraud1971}.
During the last two decades,
twisted sheaves and their derived categories have gained a renewed interest,
starting with the thesis of C{\u{a}}ld{\u{a}}raru \cite{caldararu2000},
who studied moduli problems for semi-stable sheaves on varieties,
and the work by de~Jong~\cite{dejong2004}
and Lieblich~\cite{lieblich2004, lieblich2015}
on the period--index problem for the Brauer group.
For a survey on later developments,
we refer to Lieblich~\cite[Section~4]{lieblich2017}.

Although twisted sheaves have a rather elementary description in terms of 2-cocycles
(cf.~Remark~\ref{rem-twisted-2-cocycle}),
they are most naturally viewed as sheaves on gerbes banded by diagonalizable groups
(see~Definition~\ref{def:twisted}).
Our first result is the following theorem on the decomposition of the derived category associated to such a gerbe,
generalizing a result by Lieblich~\cite[Section~2.2.4]{lieblich2004}.

\begin{mtheorem}[{see Theorem~\ref{t:chi-splitting-D-qc-X}}]
  \label{t:chi-splitting-D-qc-X-intro}
  Let $S$ be an algebraic stack and 
  $\pi\colon X \to S$ a gerbe banded by a diagonalizable group
  $\Delta$ 
  with character group $A$.
  Then taking the coproduct defines an equivalence 
  \begin{equation}
    \label{eq-equivalence-derived-intro}
    \prod_{\chi \in A} \D_{\qc, \chi}(X) \sira \D_\qc(X),
    \qquad
    (\mathcal{F}_\chi)_{\chi \in A} 
    \mapsto 
    \bigoplus_{\chi \in A} \mathcal{F}_\chi,
  \end{equation}
  of triangulated categories,
  where $\D_{\qc, \chi}(X)$ denotes the full subcategory of
  $\D_\qc(X)$ of objects with $\chi$-homogeneous quasi-coherent
  cohomology (see Definition~\ref{def-homogeneous-sheaf}).
\end{mtheorem}

In the underived setting, 
a splitting similar as in the theorem above (see Theorem~\ref{t-qc-gerbe-split})
is a consequence of the well-known fact that quasi-coherent
representations of $\Delta$ split into subrepresentations
corresponding to the characters of the group~$\Delta$.

The result by Lieblich is the decomposition in Theorem~\ref{t:chi-splitting-D-qc-X-intro}
in the special case when $S$ is a quasi-compact, separated scheme.
He proves and uses the fact that in this setting the obvious functor
\begin{equation}
\label{eq-dqcoh-vs-dqcoh}
\D(\Qcoh(X)) \to \D_\qc(X)
\end{equation}
is an equivalence of categories~\cite[Proposition~2.2.4.6]{lieblich2004}.
Note that, if the base $S$ is a scheme, the functor \eqref{eq-dqcoh-vs-dqcoh} only fails to be an equivalence
in truly pathological situations.
However, as shown by Hall--Neeman--Rydh \cite[Theorem~1.3]{hnr2018},
the functor \eqref{eq-dqcoh-vs-dqcoh} fails to be an equivalence for large classes of algebraic stacks,
including the basic case when $X$ is the classifying stack $\BB \GL_n$ of the general linear group
(see Remark~\ref{rem-trivial-case}).
In particular, this kind of argument cannot be used to obtain the splitting of the derived category
of one of the most fundamental examples of gerbes banded by $\GGm$ ---
namely the \emph{gerbe of trivializations}
$\BB \GL_n \to \BB \PGL_n$ for the universal Brauer--Severi scheme
(see Remark~\ref{rem-universal-gerbe-of-trivializations}).

In order to prove Theorem~\ref{t:chi-splitting-D-qc-X-intro},
we study the more general problem whether suitable torsion pairs in abelian categories induce
semi-orthogonal decompositions on the level of derived categories.
We achieve the following theorem of independent interest.

\begin{mtheorem}[{see Theorem~\ref{t:sod-DbBA} and Theorem~\ref{t:sod-DBA}}]
  \label{t:sod-DBA-intro}
  Let $\mathcal{B}$ be a weak Serre subcategory of an abelian
  category $\mathcal{A}$.
  Let 
  $\mathcal{T}$ and
  $\mathcal{F}$ be abelian subcategories of $\mathcal{B}$
  forming a torsion pair $(\mathcal{T}, \mathcal{F})$ in
  $\mathcal{B}$.
  Assume that 
  \begin{equation*}
    \Ext^n_{\mathcal{A}}(T, F)=0
  \end{equation*}
  for all objects $T \in \mathcal{T}$, $F \in \mathcal{F}$ and
  all integers $n$. 
  Then 
  \begin{equation*}
    \D_{\mathcal{B}}^\bd(\mathcal{A})=
    \langle \D_{\mathcal{F}}^\bd(\mathcal{A}), 
    \D_{\mathcal{T}}^\bd(\mathcal{A})\rangle
  \end{equation*}
  is a semi-orthogonal decomposition.
  Furthermore, if the inclusion $\D_{\mathcal{B}}(\mathcal{A})
  \subset \D(\mathcal{A})$ 
  satisfies some technical
  conditions on homotopy limits and colimits
  formulated precisely in the statement of Theorem~\ref{t:sod-DBA},
  then
  \begin{equation*}
    \D_{\mathcal{B}}(\mathcal{A})=
    \langle \D_{\mathcal{F}}(\mathcal{A}), 
    \D_{\mathcal{T}}(\mathcal{A})\rangle
  \end{equation*}
  is a semi-orthogonal decomposition.
\end{mtheorem}

The theory of gerbes,
hinging on the theory of stacks,
used to have a reputation of being inaccessible,
and Giraud's book on non-abelian cohomology \cite{giraud1971} introducing them is notorious for being a hard read,
even among experts.
However,
in recent years the theory of algebraic stacks has become a main stream part of algebraic geometry,
much owing to the excellent text books by Laumon--Moret-Bailly~\cite{lmb2000} and Olsson~\cite{olsson2016},
and to the Stacks Project~\cite{stacks-project}.
Moreover,
as explained in the thesis by Lieblich \cite{lieblich2004},
only a small part of Giraud's theory is actually needed to
develop a satisfactory theory of twisted sheaves.
We illustrate the effectiveness of the language of gerbes by giving a short,
simple and conceptually appealing proof of the following theorem,
which generalizes a result by Bernardara \cite[Theorem~4.1]{bernardara-BS-schemes}.

\begin{mtheorem}[{see Theorem~\ref{t:sod-brauer-severi} and
    Corollary~\ref{c:sod-brauer-severi}}]
  \label{t:sod-brauer-severi-intro}
  Let $S$ be an algebraic stack and $\pi\colon P \to S$ a
  Brauer--Severi 
  scheme of relative dimension $n$ over $S$.
  Let $\beta \in \HH^2(S_\fppf, \GGm)$ denote the Brauer class of
  $\pi$. 
  Then, for any $a \in \bZ$, the category $\D_\qc(P)$ admits a
  semi-orthogonal 
  decomposition
  \begin{equation}
    \label{eq:SOD-Dpf}
    \D_\qc(P) = \langle \mathcal{D}_a, \ldots, \mathcal{D}_{a+n}
    \rangle 
  \end{equation}
  into right admissible subcategories $\mathcal{D}_i$,
  where the category $\mathcal{D}_i$ is equivalent to 
  the category $\D_{\qc,\beta^i}(S)$
  of $\beta^i$-twisted complexes on $S$ (see
  Definition~\ref{def:derived-twisted} and Remark~\ref{r:image-twisted}). 
  
  Similarly, we have corresponding decompositions
  for the category of perfect complexes, the category of locally bounded pseudo-coherent complexes
  and the singularity category.  
  In particular, we also have such a decomposition for $\D^b_\co(P)$ when
  $S$ is Noetherian.
\end{mtheorem}

Bernardara's result is the semi-orthogonal decomposition 
\eqref{eq:SOD-Dpf} in the case when $S$ is a Noetherian, separated
scheme with the property that any pair of points in $S$ is
contained in an open affine subscheme (cf.\ Remark~\ref{r:bernardara}).
Apart from our Theorem~\ref{t:chi-splitting-D-qc-X-intro},
the main ingredient in our proof
of Theorem~\ref{t:sod-brauer-severi-intro} is a classical
result by Orlov regarding a semi-orthogonal decomposition of the derived
category of a projectivized vector bundle \cite[Theorem~2.6]{orlov1992},
which we previously have generalized to algebraic stacks using the technique of \emph{conservative descent}
\cite[Theorem~6.7, Corollary~6.8]{BS-conservative}.

While finishing the manuscript, we were informed by Brown and~Moulinos of a result \cite[Theorem 3.1]{bm2019}
which is similar to our Theorem~\ref{t:sod-brauer-severi-intro}.

\subsection*{Outline}
In Section~\ref{sec:torsion-pairs-decomp},
we work mostly in the general setting of derived categories of abelian categories.
This section contains the proof of Theorem~\ref{t:sod-DBA-intro}.
In Section~\ref{sec:gerbes-and-twisted} and Section~\ref{sec:quasi-coher-sheav},
we recall and collect some basic facts on gerbes,
bandings and twisted sheaves which are scattered in the literature.
We expect most of these results to be well known to experts.
In Section~\ref{sec:decomp-deriv-categ}, we turn our attention to derived categories of gerbes and twisted sheaves.
We generalize some of the fundamental results by C{\u{a}}ld{\u{a}}raru and Lieblich \cite{caldararu2000, lieblich2004}
to the case when we work over an algebraic stack.
Most notably we prove Theorem~\ref{t:chi-splitting-D-qc-X-intro}.
Finally, in the last section,
we give a brief summary on Giraud's treatment of the Brauer group \cite[Section~V.4]{giraud1971},
and use the theory to prove Theorem~\ref{t:sod-brauer-severi-intro}.

\subsection*{Acknowledgments}
The first named author was partially supported by the Danish National
Research Foundation through the Niels Bohr Professorship of Lars Hesselholt,
by the Max Planck Institute for Mathematics in Bonn,
and by the DFG through SFB/TR 45.
The second named author was partially supported by the DFG through a postdoctoral fellowship and through SFB/TR 45.

\subsection*{Notation and conventions}
\label{sec:conventions}

If $\mathcal{A}$ is an abelian category, we write
$\D(\mathcal{A})$ for its unbounded derived category. 
We write $\D^{\leq n}(\mathcal{A})$ and 
$\D^{\geq n}(\mathcal{A})$ for its full subcategories of objects
whose cohomology is concentrated in the indicated degrees, where
$n \in \bZ$.
The (intelligent) truncation functors on $\D(\mathcal{A})$ are
denoted $\tau_{\leq n}$ and $\tau_{\geq n}$. 
If $E$ and $F$ are objects of $\D(\mathcal{A})$, we 
sometimes abbreviate
$\Ext^n_\mathcal{A}(E,F):=\Hom_{\D(\mathcal{A})}(E, \Sigma^nF)$.

By a \emph{diagonalizable group} we mean a diagonalizable group scheme
over $\Spec \bZ$ that is isomorphic to the spectrum of the group
ring of a finitely generated abelian group.

We use the definitions of algebraic space and algebraic stack given in the stacks project
\cite[\sptag{025Y}, \sptag{026O}]{stacks-project}.
In particular, we do not impose any separatedness conditions. 
The algebraic stacks form a 2-category where each 2-morphism is invertible.
We will follow the common practice to often suppress 2-categorical details.
For instance, we will usually simply write morphism of stacks instead of 1-morphism
of stacks, subcategory of the category of stacks instead of 2-subcategory etc.

For the theory of quasi-coherent modules and derived categories in the context of algebraic stacks,
we basically follow the approach of Laumon--Moret-Bailly~\cite[Sections~12 and~13]{lmb2000},
which we now briefly recapitulate.
Let $X$ be an algebraic stack.
We denote the topos of sheaves on the big \emph{fppf} site over $X$ by $X_\fppf$
and the topos of sheaves on the \emph{lisse--étale} site over $X$ by $X_\liset$.
For $\tau \in \{\fppf, \liset\}$,
we let $\Mod(X_\tau, \mathcal{O}_X)$ denote category of $\mathcal{O}_X$-modules in $X_\tau$,
and $\Qcoh(X_\tau, \mathcal{O}_X)$ the full subcategory of quasi-coherent modules
(see~\cite[\sptag{03DL}]{stacks-project}).
\begin{remark}
\label{rem-coherator}
The categories $\Qcoh(X_\tau, \mathcal{O}_X)$ and $\Mod(X_\tau, \mathcal{O}_X)$ are Grothendieck abelian categories
\cite[\sptag{0781}]{stacks-project}.
In particular,
since the inclusion functor $\Qcoh(X_\tau, \mathcal{O}_X) \hookrightarrow \Mod(X_\tau, \mathcal{O}_X)$ is colimit preserving,
it admits a right adjoint.
This right adjoint is called the \emph{quasi-coherator}.
\end{remark}

The construction $(-)_\tau$ functorially associates an adjoint pair 
\begin{equation}
\label{eq-functorial-topos}
f^*\colon Y_\tau \rightleftarrows X_\tau \colon f_*
\end{equation}
of functors to any morphism $f\colon X \to Y$ of algebraic stacks.
Moreover, the functor $f^*$ preserves finite products.
In particular, the adjunction \eqref{eq-functorial-topos} induces an adjunction
\begin{equation}
\label{eq-functorial-module}
f^*\colon \Mod(Y_\tau, \mathcal{O}_Y) \rightleftarrows \Mod(X_\tau, \mathcal{O}_X) \colon f_*.
\end{equation}
Moreover, the pull-back functor $f^*$ preserves quasi-coherence.
The functor
\begin{equation}
\label{eq-qc-push-forward}
f_*\colon \Qcoh(X_\tau, \mathcal{O}_X) \to \Qcoh(Y_\tau, \mathcal{O}_Y)
\end{equation}
is defined as the right adjoint of the restriction of $f^*$ to
quasi-coherent sheaves.
The existence of such a right adjoint is guaranteed by the quasi-coherator
(see~Remark~\ref{rem-coherator}).
\begin{remark}
\label{rem-qc-pf}
If $\tau = \liset$, then the functor \eqref{eq-qc-push-forward} is the restriction of the push-forward for sheaves of modules
provided that $f$ is quasi-compact and quasi-separated,
but not in general.
If $\tau = \fppf$, then \eqref{eq-qc-push-forward} seldom is the restriction of the push-forward for sheaves of modules.
\end{remark}
The category $\Mod(X_\tau, \mathcal{O}_X)$ has a closed symmetric monoidal
structure with operations given by the usual tensor product  and the sheaf hom functor.
The tensor operation preserves quasi-coherence and is preserved by pull-backs.
In particular, we get an induced symmetric monoidal structure on $\Qcoh(X_\tau, \mathcal{O}_X)$.

\begin{remark}
\label{rem-qc-hom}
The symmetric monoidal structure on $\Qcoh(X_\tau, \mathcal{O}_X)$ induced by the tensor product
is also closed.
This follows from the existence of the quasi-coherator (see~Remark~\ref{rem-coherator}).
For $\tau = \liset$,
the internal hom functor coincides with the sheaf hom functor provided that the first argument is of finite presentation, but not in general.
\end{remark}

There is an obvious restriction functor $X_\fppf \to X_\liset$,
which is compatible with push-forward.
This induces a monoidal equivalence
\begin{equation}
\label{eq-change-site}
\Qcoh(X_\fppf, \mathcal{O}_X) \sira \Qcoh(X_\liset, \mathcal{O}_X),
\end{equation}
functorial in $X$ (see~\cite[\sptag{07B1}]{stacks-project}).

We denote the derived category $\D(\Mod(X_\liset, \mathcal{O}_X))$ simply by $\D(X)$.
The category $\D_\qc(X)$ is defined as the full subcategory of $\D(X)$ consisting of
complexes with quasi-coherent cohomology sheaves.
The derived tensor product on $\D(X)$ preserves objects in $\D_\qc(X)$,
giving the latter category a symmetric monoidal structure, which is closed.
\begin{remark}
\label{rem-dqc-hom}
The internal hom functor on $\D_\qc(X)$ is the restriction of the derived sheaf hom functor provided
that the first argument is perfect, but not in general
(see~\cite[Section~1.2]{hr2017}). 
\end{remark}
Given an arbitrary morphism $f\colon X \to Y$ of algebraic stacks,
we have an adjoint pair of functors
\begin{equation}
\label{eq-pb-pf}
f^*\colon \D_\qc(Y) \to \D_\qc(X), \qquad f_*\colon \D_\qc(Y) \to \D_\qc(X),
\end{equation}
and the formation of such pairs is functorial.
\begin{remark}
\label{rem-derived-pb}
The functor $f^*$ in \eqref{eq-pb-pf} coincides with the derived pull-back between the ambient categories provided that
$f$ is smooth.
For general $f$, the construction of $f^*$ is somewhat technical owing to the fact that the adjunction
\eqref{eq-functorial-topos}, for $\tau = \liset$,
does not form a morphism of topoi,
as was first noted by Behrend~\cite[Warning~5.3.12]{behrend-derived-l-adic-stacks}.
The construction is worked out by Olsson~\cite{olsson2007} in the bounded case and by Laszlo--Olsson~\cite[Example~2.2.5]{lo2008} in the unbounded case.
We refer to \cite[Section~1]{hr2017} for a concise summary on this.
\end{remark}
\begin{remark}
\label{rem-derived-pf}
The push-forward $f_*$ in \eqref{eq-pb-pf} coincides with the derived push-forward provided that
$f$ is \emph{concentrated} \cite[Definition~2.4, Theorem~2.6(2)]{hr2017},
but not in general.
\end{remark}

\begin{remark}
It is also possible to construct $\D_\qc(X)$ using the topos $X_\fppf$.
This approach is taken in the stacks project \cite[\sptag{07B6}]{stacks-project}.
For $\tau = \fppf$ the adjunction \eqref{eq-functorial-topos} does give a morphism of topoi,
making the construction of the pull-back in \eqref{eq-pb-pf} easier.
On the other hand the inclusion $\Qcoh(X_\fppf, \mathcal{O}_X) \subset \Mod(X_\fppf, \mathcal{O}_X)$
is not exact \cite[\sptag{06WU}]{stacks-project},
making the actual construction of $\D_\qc(X)$ quite technical.
We will not use this point of view in the present article.
\end{remark}

\begin{notation}
\label{not-sheaf-op}
We conclude this section by summarizing our notational conventions regarding sheaves on stacks.
In the discussion below, we let $\tau \in \{\fppf, \liset\}$.
\begin{enumerate}
\item
We simply write $\Qcoh(X)$ for any of the categories $\Qcoh(X_\tau, \mathcal{O}_X)$
when no distinction is necessary.
This is motivated by the functorial equivalence~\eqref{eq-change-site}.
\item
The monoidal operations on the categories
$$
\Qcoh(X_\tau, \mathcal{O}_X), \qquad
\Mod(X_\tau, \mathcal{O}_X), \qquad
\D(X), \qquad
\D_\qc(X)
$$
are always denoted $-\otimes-$ and $\sheafHom(-,-)$, respectively.
In particular, we omit any derived decorations.
The precise category we are working on will always be inferable from context.
Note that in the cases $\Qcoh(X_\tau, \mathcal{O}_X)$ and~$\D_\qc(X)$
the functor $\sheafHom$ does not always coincide with $\sheafHom$ on
the ambient category (cf.~Remark~\ref{rem-qc-hom} and~Remark~\ref{rem-dqc-hom}). 
\item
The pull-backs and push-forwards along a morphism $f\colon X \to Y$ are always denoted by $f^*$ and~$f_*$,
respectively.
Which of the categories
$$
(-)_\tau,\qquad
\Mod((-)_\tau, \mathcal{O}_{(-)}), \qquad
\Qcoh((-)_\tau, \mathcal{O}_{(-)}),
$$
$$
\D(-), \qquad
\D_\qc(-),
$$
we are working with will always be inferable from context.
In particular, 
we do not use any derived decorations or any particular decorations for
push-forward of quasi-coherent sheaves (cf.~Remark~\ref{rem-qc-pf} and~Remark~\ref{rem-derived-pf}).
\end{enumerate}
\end{notation}


\section{Torsion pairs and
decompositions of derived categories}
\label{sec:torsion-pairs-decomp}

We give criteria ensuring that a torsion pair in an abelian
category gives rise to a semi-orthogonal decomposition on the
level of derived categories. Our main assumption on the torsion
pair is that its two components are abelian subcategories of the
ambient abelian category. 
For bounded derived categories this is straightforward (see
Theorem~\ref{t:sod-DbBA}) as soon as some foundational results
for such torsion pairs are established (see
Proposition~\ref{p:abelian-torsion-pair}). Our result for
unbounded derived categories (see Theorem~\ref{t:sod-DBA}) needs
some technical assumptions on effectiveness of inverse and direct
truncation systems introduced in
Section~\ref{sec:notes-htpy-limits}. Fortunately, these
assumptions are satisfied for the derived category $\D_\qc(X)$ of
an algebraic stack (see Example~\ref{ex:ringed-topos-effective}
and Proposition~\ref{p:Dqc-truncation-complete}).

\subsection{Torsion pairs}
\label{sec:torsion-pairs}

We remind the reader of the notion of a \emph{torsion pair}. 
Sometimes
the terminology \emph{torsion theory} is used in the
literature. Standard references are
\cite[I.1]{beligiannis-reiten-torsion-theories},
\cite[1.12]{borceux-cat-2}, \cite{dickson-torsion}. 

\begin{definition}
  \label{d:torsion-pair}
  A \define{torsion pair} in an abelian category $\mathcal{B}$ is
  a pair 
  $(\mathcal{T}, \mathcal{F})$
  of strictly full subcategories of $\mathcal{B}$ such that the
  following two conditions hold:
  \begin{enumerate}[label=(TP{\arabic*})]
  \item 
    \label{enum:TP-ses}
    Any object $B \in \mathcal{B}$ fits into a short exact
    sequence
    \begin{equation*}
      0 \ra T \ra B \ra F \ra 0
    \end{equation*}
    with $T \in \mathcal{T}$ and $F \in \mathcal{F}$. 
  \item 
    \label{enum:TP-Hom-TF=0}
    We have $\Hom_\mathcal{B}(T, F)=0$ for all objects $T \in
    \mathcal{T}$ and $F \in \mathcal{F}$.
  \end{enumerate}
\end{definition}

\begin{remark}
  \label{rem:adjoints-f-and-t}
  Let $(\mathcal{T}, \mathcal{F})$
  be a torsion pair in an abelian category $\mathcal{B}$. 
  The axioms of a torsion pair immediately imply that 
  the short exact sequence in \ref{enum:TP-ses} is unique up to
  unique isomorphism and functorial in $B \in \mathcal{B}$.
  In particular, there are
  functors $t \colon \mathcal{B} \ra \mathcal{T}$ and $f \colon
  \mathcal{B} \ra \mathcal{F}$ together with 
  natural
  transformations
  $t \ra \id$ and $\id \ra f$ giving rise to short exact
  sequences 
  \begin{equation*}
    0 \ra t(B) \ra B \ra f(B) \ra 0,
  \end{equation*}
  for each object $B \in \mathcal{B}$. The object $B$ is in
  $\mathcal{T}$ (resp.\ $\mathcal{F})$ if and only if $f(B)=0$
  (resp.\ $t(B)=0$). 
  The functor $t$ is right adjoint to the inclusion functor
  $\mathcal{T} \ra \mathcal{B}$, and the functor $f$ is left
  adjoint to the inclusion functor 
  $\mathcal{F} \ra \mathcal{B}$.
\end{remark}

Recall that an \emph{abelian subcategory} of an abelian category
is a strictly full subcategory which is abelian and whose
inclusion functor is exact. 
A \emph{weak Serre subcategory}
is an abelian subcategory which is closed under
extensions in the ambient category.
If it is even closed under taking subobjects
and quotients in the ambient category, it is called a \emph{Serre
  subcategory}.  

\begin{remark}
  In the following we are only interested in torsion pairs 
  $(\mathcal{T}, \mathcal{F})$ where both
  $\mathcal{T}$ and $\mathcal{F}$ are abelian subcategories. 
  These torsion pairs are traditionally called
  \emph{hereditary} and \emph{cohereditary}. We do not use this
  terminology. 
\end{remark}

\begin{example}
  \label{e:A2-reps}
  Let $\mathcal{B}$ be the category of
  representations of the $A_2$-quiver $1 \ra 2$ over
  some ring. If $\mathcal{T}$ (resp.\ $\mathcal{F}$) denotes the
  subcategory of
  representations supported at the vertex $2$ (resp.\ $1$)
  then $(\mathcal{T}, \mathcal{F})$ is clearly a torsion pair in
  $\mathcal{B}$. Moreover, both $\mathcal{T}$ and
  $\mathcal{F}$ are abelian subcategories of $\mathcal{B}$.
\end{example}

\begin{proposition}
  \label{p:abelian-torsion-pair}
  Let $(\mathcal{T}, \mathcal{F})$
  be a torsion pair in an abelian category $\mathcal{B}$. If both
  $\mathcal{T}$ and $\mathcal{F}$ are abelian subcategories of
  $\mathcal{B}$, then the following statements are true.
  \begin{enumerate}
  \item 
    \label{enum:Hom-FT}
    $\Hom_\mathcal{B}(F, T)=0$ for all objects $F \in
    \mathcal{F}$ and $T \in 
    \mathcal{T}$.
  \item 
    \label{enum:adjoints-exact}
    Both the right adjoint to the inclusion functor
    $\mathcal{T} \ra \mathcal{B}$ and 
    the left adjoint to the inclusion functor
    $\mathcal{F} \ra \mathcal{B}$ are exact,
    i.e., in the notation of Remark~\ref{rem:adjoints-f-and-t},
    both functors $t$ and $f$ are exact.
  \item 
    \label{enum:serre-subcats}
    Both $\mathcal{T}$ and $\mathcal{F}$ are Serre
    subcategories of $\mathcal{B}$.
  \item 
    \label{enum:Ext-TF}
    $\Ext^n_\mathcal{B}(T,F)=0$ for all objects
    $T \in \mathcal{T}$, $F \in \mathcal{F}$ and all integers $n
    \in \bZ$. 
  \end{enumerate}
\end{proposition}

\begin{remark}
  The most interesting part of
  Proposition~\ref{p:abelian-torsion-pair} 
  is certainly the $\Ext$-vanishing result in
  part~\ref{enum:Ext-TF}.
  After writing down its proof we learned that Brion has recently
  obtained the same result \cite[Lemma 2.3]{brion2018}.  
  For the convenience of the reader we give a full proof of
  Proposition~\ref{p:abelian-torsion-pair} even though parts
  \ref{enum:Hom-FT}, \ref{enum:adjoints-exact},
  \ref{enum:serre-subcats} are well-known or straightforward.
\end{remark}

\begin{proof}
  \ref{enum:Hom-FT}
  Let $\phi \colon F \ra T$ be a morphism from $F
  \in \mathcal{F}$ to $T \in \mathcal{T}$. 
  By \ref{enum:TP-ses}, we may
  put $\kernel(\phi)$ into a short exact sequence
  $0 \ra T' \ra \kernel(\phi) \ra F' \ra 0$ with $T' \in
  \mathcal{T}$ and $F' \in \mathcal{F}$. Since the composition of
  monomorphisms
  $T' \ra \kernel(\phi) \ra F$ is zero,
  by \ref{enum:TP-Hom-TF=0}, we have $T'=0$ and
  $\kernel(\phi) \in \mathcal{F}$.
  Similarly one shows $\cok(\phi) \in
  \mathcal{T}$. But then $\im(\phi)$ is the cokernel of
  $\kernel(\phi) \ra F$ and the
  kernel of $T \ra \cok(\phi)$ and hence is in $\mathcal{F} \cap
  \mathcal{T}$ since both $\mathcal{F}$ and $\mathcal{T}$ are
  abelian subcategories, i.e., $\im(\phi)=0$ and $\phi=0$.

  \ref{enum:adjoints-exact}
  We use the notation of
  Remark~\ref{rem:adjoints-f-and-t}. Clearly, $t$ and $f$ are
  additive. 
  Let $0 \ra B_1 \ra B_2 \ra B_3 \ra 0$ be a short exact
  sequence in $\mathcal{B}$. Then we obtain a commutative diagram
  with exact rows
  \begin{equation*}
    \xymatrix{
      {0} \ar[r]
      &
      {t(B_3)} \ar[r]
      &
      {B_3} \ar[r]
      &
      {f(B_3)} \ar[r]
      &
      {0}
      \\
      {0} \ar[r]
      &
      {t(B_2)} \ar[r] \ar[u]
      &
      {B_2} \ar[r] \ar[u]
      &
      {f(B_2)} \ar[r] \ar[u]
      &
      {0}
      \\
      {0} \ar[r]
      &
      {t(B_1)} \ar[r] \ar[u]
      &
      {B_1} \ar[r] \ar[u]
      &
      {f(B_1)} \ar[r] \ar[u]
      &
      {0.}
    }
  \end{equation*}
  We view the columns as complexes in $\mathcal{B}$ by adding
  zeroes.
  Then the diagram is a short exact sequence of
  complexes and gives rise to a long exact sequence on cohomology.
  Since the middle column is exact, since the cohomology objects
  of the left 
  (resp.\ right)
  column are in $\mathcal{T}$ (resp.\ $\mathcal{F})$, and since there
  are no nonzero morphisms 
  from $\mathcal{F}$ to $\mathcal{T}$, by \ref{enum:Hom-FT},
  the other two columns are also exact. This shows that $t$ and
  $f$ are exact.

  \ref{enum:serre-subcats}
  Let $F \ra B \ra F'$ be an exact sequence in $\mathcal{B}$ where
  $F$ and $F'$ are objects of 
  $\mathcal{F}$.
  To see that $\mathcal{F}$ is a Serre subcategory of
  $\mathcal{B}$ we need to see that $B \in \mathcal{F}$.
  The functor $t \colon \mathcal{B} \ra
  \mathcal{T}$ is exact, by 
  \ref{enum:adjoints-exact},
  and hence yields
  the exact sequence 
  $t(F) \ra t(B) \ra t(F')$ in $\mathcal{T}$.
  Since $t(F)=0=t(F')$ we obtain $t(B)=0$, i.e., $B \in
  \mathcal{F}$.

  The same argument proves that 
  $\mathcal{T}$ is a Serre subcategory of
  $\mathcal{B}$.

  \ref{enum:Ext-TF}
  The claim is trivially true for $n<0$. For 
  $n=0$ it is true 
  by axiom~\ref{enum:TP-Hom-TF=0}
  of a torsion pair.
  
  For $n=1$ we need to prove that 
  $\Ext^1_\mathcal{B}(T,F)=0$
  for $T \in \mathcal{T}$ and $F \in \mathcal{F}$. 
  Using Yoneda extensions (and
  \cite[\sptag{06XU}]{stacks-project})
  it is enough to show that any short exact sequence
  $0 \ra F \ra B \ra T \ra 0$ in $\mathcal{B}$ splits.
  Since $t$ is exact, by \ref{enum:adjoints-exact}, and $t(F)=0$,
  we obtain a morphism
  of short exact sequences
  \begin{equation*}
    \xymatrix{
      {0} \ar[r] 
      &
      {0} \ar[r] \ar[d]
      &
      {t(B)} \ar[r]^-{\sim} \ar[d]
      &
      {t(T)} \ar[r] \ar[d]^-{\sim}
      &
      {0}
      \\
      {0} \ar[r] 
      &
      {F} \ar[r]
      &
      {B} \ar[r]
      &
      {T} \ar[r]
      &
      {0}
    }
  \end{equation*}
  in $\mathcal{B}$; the vertical arrow on the right is an
  isomorphism since $T \in \mathcal{T}$ and hence $f(T)=0$.
  This shows that our sequence splits. (Since
  $\Hom_\mathcal{B}(T,F)=0$, the splitting is in fact unique.)

  Now let $n \geq 2$.
  Let $f \in \Ext_\mathcal{B}^n(T,F)$, i.e., $f \colon T \ra
  \Sigma^n F$ is a morphism in
  $\D(\mathcal{B})$.
  The $n$-extension $f$ can be written as the composition
  of a 1-extension and an $(n-1)$-extension: this is easy to
  prove directly or follows alternatively using
  Yoneda extensions and \cite[\sptag{06XU}]{stacks-project}.
  Anyhow, there is an object $B \in \mathcal{B}$ 
  such that $f$ is the composition
  \begin{equation*}
    f \colon T \xra{g} \Sigma^{n-1}B \xra{\Sigma^{n-1}h} \Sigma^n F,
  \end{equation*}
  where $g$ and $h \colon B \ra \Sigma F$ are 
  morphisms in $\D(\mathcal{B})$.
  The long exact $\Ext_\mathcal{B}(-,F)$-sequence obtained from
  the short 
  exact sequence $0 \ra t(B) \ra B \xra{\pi} f(B) \ra 0$ yields an
  isomorphism
  \begin{equation*}
    \pi^* \colon 
    \Ext^1_\mathcal{B}(f(B), F) \sira
    \Ext^1_\mathcal{B}(B, F) 
  \end{equation*}
  since we already know that 
  $\Ext^1_\mathcal{B}(t(B), F)$ and
  $\Hom_\mathcal{B}(t(B), F)$ vanish.
  Hence $h$ has the form $h=h' \circ \pi$ for a (unique) morphism
  $h' \colon f(B) \ra \Sigma F$ in $\D(\mathcal{B})$.
  Then
  \begin{equation*}
    f=(\Sigma^{n-1}h) \circ g = (\Sigma^{n-1}h') \circ
    (\Sigma^{n-1}\pi) \circ g. 
  \end{equation*}
  By induction, the $(n-1)$-extension $(\Sigma^{n-1}\pi) \circ g
  \colon T 
  \ra  \Sigma^{n-1}f(B)$ vanishes. Hence $f=0$.
\end{proof}

\subsection{Decompositions of bounded derived categories}
\label{sec:decomp-bound-deriv}

If $\mathcal{S}$ is a weak Serre subcategory of an abelian
category $\mathcal{B}$, the full subcategory of $\D(\mathcal{B})$
of objects with cohomology objects in
$\mathcal{S}$ is denoted
$\D_{\mathcal{S}}(\mathcal{B})$. 
Let 
$\D_{\mathcal{S}}^-(\mathcal{B})$,
$\D_{\mathcal{S}}^+(\mathcal{B})$,
and~$\D_{\mathcal{S}}^\bd(\mathcal{B})$ denote the full subcategories
of $\D_{\mathcal{S}}(\mathcal{B})$
of objects whose cohomology is bounded above, bounded below,
and bounded, respectively. 
All these categories are strictly full triangulated
subcategories of $\D(\mathcal{B})$, 
see, e.g., \cite[\sptag{06UQ}]{stacks-project}.
We refer the reader to \cite[Subsection~5.1]{BS-conservative}
for the notion of a semi-orthogonal decomposition.

\begin{theorem}
  \label{t:sod-DbBA}
  Let $\mathcal{B}$ be a weak Serre subcategory of an abelian
  category $\mathcal{A}$.
  Let 
  $\mathcal{T}$ and
  $\mathcal{F}$ be abelian subcategories of $\mathcal{B}$
  forming a torsion pair $(\mathcal{T}, \mathcal{F})$ in
  $\mathcal{B}$.
  Assume that 
  \begin{equation}
    \label{eq:DA-T-F}
    \Ext^n_{\mathcal{A}}(T, F)=0
  \end{equation}
  for all objects $T \in \mathcal{T}$, $F \in \mathcal{F}$ and
  all integers $n$. 
  Then 
  \begin{equation*}
    \D_{\mathcal{B}}^\bd(\mathcal{A})=
    \langle \D_{\mathcal{F}}^\bd(\mathcal{A}), 
    \D_{\mathcal{T}}^\bd(\mathcal{A})\rangle
  \end{equation*}
  is a semi-orthogonal decomposition.
\end{theorem}

\begin{remark}
Condition~\eqref{eq:DA-T-F} is automatically satisfied
in the special case when $\mathcal{B}=\mathcal{A}$,
by Proposition~\ref{p:abelian-torsion-pair}.\ref{enum:Ext-TF}.
In general, the condition is automatically satisfied for $n \leq 1$
by the same proposition, since $\mathcal{B}$ is closed under extensions.
\end{remark}

\begin{proof}
  By Proposition~\ref{p:abelian-torsion-pair}.\ref{enum:serre-subcats},
  $\mathcal{F}$ and $\mathcal{T}$ are Serre
  subcategories of $\mathcal{B}$, hence weak Serre
  subcategories of $\mathcal{A}$, and 
  $\D_{\mathcal{F}}^\bd(\mathcal{A})$ and  
  $\D_{\mathcal{T}}^\bd(\mathcal{A})$ are strictly full triangulated
  subcategories of $\D_\mathcal{B}^\bd(\mathcal{A})$.

  Repeated use of truncation shows that 
  $\D_{\mathcal{B}}^\bd(\mathcal{A})$ coincides with its smallest strictly
  full triangulated subcategory containing $\mathcal{B}$.
  Axiom~\ref{enum:TP-ses} shows that it is
  also its smallest such subcategory containing both
  $\mathcal{F}$ and $\mathcal{T}$.
  In particular, $\D_{\mathcal{B}}^\bd(\mathcal{A})$ is generated,
  as a triangulated category, by its subcategories
  $\D_{\mathcal{F}}^\bd(\mathcal{A})$ 
  and $\D_{\mathcal{T}}^\bd(\mathcal{A})$.

  By truncation and condition \eqref{eq:DA-T-F},
  there are no nonzero morphisms from 
  $\D_{\mathcal{T}}^\bd(\mathcal{A})$ to 
  $\D_{\mathcal{F}}^\bd(\mathcal{A})$.
\end{proof}
 
\subsection{Some notes on homotopy limits of truncations}
\label{sec:notes-htpy-limits}

Our aim is to state and prove Theorem~\ref{t:sod-DBA} 
in the following subsection~\ref{sec:decomp-unbo-deriv}
which is
the analog  
of Theorem~\ref{t:sod-DbBA} for unbounded
derived categories. In order to to this, 
we need some terminology and some basic observations.

We refer the reader to
\cite[Definition~1.6.4]{neeman-tricat} for the notion of a
homotopy colimit. The definition of a homotopy limit is dual.

Let $\mathcal{A}$ be an abelian category.
We propose the following terminology. 

\begin{definition}
\label{d:truncation-system}
\rule{0mm}{1mm}
\begin{enumerate}
\item
An inverse system 
\begin{equation*}
  (F_n)_{n \in \mathbb{N}} = 
  (\dots \ra F_{n+1} \ra F_n \ra \dots \ra F_2 \ra F_1 \ra F_0)
\end{equation*}
in $\D(\mathcal{A})$
is an \define{inverse truncation system} if
for each $n \in \mathbb{N}$
the object $F_n$ is in
$\D^{\geq -n}(\mathcal{A})$
and the map
$\tau_{\geq -n}F_{n+1} \to F_n$ induced by the transition
morphism is an isomorphism
(cf.\ Remark~\ref{rem:ITS-alternative} for a reformulation).
\item
If  $\D(\mathcal{A})$ has countable products, 
an inverse truncation system 
$(F_n)_{n \in \mathbb{N}}$ in $\D(\mathcal{A})$
is \define{effective} if
the induced maps $\HH^p(\holim_n F_n) \to
\lim_n\HH^p(F_n)$,
are isomorphisms for all $p \in \bZ$.
Our assumption ensures the existence of homotopy limits,
and all limits $\lim_n\HH^p(F_n)$
exist by 
Remark~\ref{rem:ITS-alternative} below.  
\end{enumerate}
Dually, we define a \emph{direct truncation system} in the
obvious way. 
If $\D(\mathcal{A})$ has countable coproducts we can 
talk about 
\emph{effectiveness} of direct truncation systems.
\end{definition}

\begin{example}
  \label{ex:ITS}
  Given any object $F$ in $\D(\mathcal{A})$, we get an inverse
  truncation system $(\tau_{\geq-n}F)_{n \in \mathbb{N}}$ whose
  transition morphisms are the natural maps between the truncations.
  Similarly, we get a direct truncation system 
  $(\tau_{\leq n}F)_{n \in \mathbb{N}}$.
\end{example}

\begin{definition}
  \label{d:standardITS}
  An inverse (resp.\ direct) truncation system is called
  \define{standard} if it is 
  isomorphic to an inverse (resp.\ direct) truncation system as
  in Example~\ref{ex:ITS}. 
\end{definition}

\begin{remark}
  Informally, one may think of an arbitrary
  inverse truncation system as obtained from a
  possibly non-existing object by truncation. 
  Effectiveness says that the natural candidate for such an
  object, the homotopy limit,
  indeed does the job. We explain this in
  Remark~\ref{rem:effective} below.
\end{remark}

\begin{remark}
  \label{rem:ITS-alternative}
  An inverse system $(F_n)_{n \in \mathbb{N}}$ in 
  $\D(\mathcal{A})$ is an inverse truncation system if and only
  if $\HH^p(F_n) = 0$ for all $p < -n$ and the maps
  $\HH^p(F_{n+1}) \ra \HH^p(F_n)$ induced by the
  transition maps are 
  isomorphisms for all $p \geq -n$.
  This just means that the induced inverse systems 
  $(\HH^p(F_{n}))_{n \in \mathbb{N}}$ look as follows.  
  \begin{equation*}
    \xymatrix@R0em{
      {\ldots} \ar[r]^-{\sim} 
      & {\HH^1(F_2)} \ar[r]^{\sim}
      & {\HH^1(F_1)} \ar[r]^{\sim}
      & {\HH^1(F_0)}
      \\
      {\ldots} \ar[r]^-{\sim} 
      & {\HH^0(F_2)} \ar[r]^{\sim}
      & {\HH^0(F_1)} \ar[r]^{\sim}
      & {\HH^0(F_0)}
      \\
      {\ldots} \ar[r]^-{\sim} 
      & {\HH^{-1}(F_2)} \ar[r]^{\sim}
      & {\HH^{-1}(F_1)} \ar[r]
      & {0}
      \\
      {\ldots} \ar[r]^-{\sim} 
      & {\HH^{-2}(F_2)} \ar[r]
      & {0} \ar[r]
      & {0}
     }
  \end{equation*}
  In particular, 
  then all limits 
  $\lim_n\HH^p(F_n)$ exist in $\mathcal{A}$.
\end{remark}

\begin{remark}
  \label{rem:effective}
  Assume that $\D(\mathcal{A})$ has countable products.
  Let $G:=\holim_n F_n$ be a homotopy limit of 
  an inverse truncation system $(F_n)_{n \in \bN}$.
  It comes together with morphisms
  $G \ra F_n$, for all $n \in \bZ$.
  Since $F_n \in \D^{\geq -n}(\mathcal{A})$ they come from unique
  morphisms  
  $\tau_{\geq -n}G \ra F_n$.
  Then, as a consequence
  of 
  Remark~\ref{rem:ITS-alternative},
  the following conditions are equivalent.
  \begin{enumerate}
  \item The inverse truncation system $(F_n)_{n \in \bN}$ is
    effective.
  \item 
    For all $p \in \bZ$ there is an integer $n \geq -p$ such that
    $\HH^p(G) \to
    \HH^p(F_n)$ is an isomorphism; this condition is then
    automatically 
    true for all 
    integers $n \geq -p$.
  \item The morphism
    $(\tau_{\geq -n} G)_{n \in \bN} \ra (F_n)_{n \in
      \bN}$ of inverse truncation systems is an isomorphism.
  \end{enumerate}
  In particular, any effective inverse truncation system is
  standard in the sense of Definition~\ref{d:standardITS}; more
  precisely, up to canonical isomorphism
  it is obtained by truncation from its homotopy limit.
\end{remark}

We address the natural question whether an object is the
homotopy limit of its truncations.  

\begin{lemma}[Objects as homotopy limits of their truncations]
  \label{l:holim-truncations}
  Let $F$ be an object of $\D(\mathcal{A})$ and assume that
  $\D(\mathcal{A})$ has countable products. 
  If the inverse truncation system $(\tau_{\geq -n}F)_{n \in \bN}$ is
  effective there is a triangle  
  \begin{equation*}
    F \ra \prod_{n \in \bN} \tau_{\geq -n} F \ra \prod_{n \in
      \bN} \tau_{\geq -n} F \ra \Sigma F
  \end{equation*}
  in $\D(\mathcal{A})$ which exhibits $F$ as a
  homotopy limit of its truncations and whose first morphism
  is induced by the truncation maps $F \ra \tau_{\geq -n}F$.
\end{lemma}

\begin{proof}
  Let $G$ be a homotopy limit 
  of the given inverse system
  with defining triangle
  \begin{equation*}
    G \ra \prod \tau_{\geq -n}F \ra \prod \tau_{\geq -n}F \ra
    \Sigma G.
  \end{equation*}
  The obvious map from $F$ to the second term of this triangle
  comes from a 
  (possibly non-unique) morphism $\gamma \colon F \ra G$ to its
  first term which makes the upper left triangle in the
  following diagram commutative.
  Since the map $G \ra \tau_{\geq -n}F$ factors through the
  truncation map $G \ra \tau_{\geq -n}G$, we obtain the
  commutative diagram
  \begin{equation*}
    \xymatrix{
      {F} \ar[r] \ar[d]^-{\gamma} & {\tau_{\geq -n}F}\\
      {G} \ar[r] \ar[ur] & {\tau_{\geq -n} G.} \ar[u]
    }
  \end{equation*}
  Its right vertical arrow is an isomorphism
  since the inverse truncation system $(\tau_{\geq -n}F)_{n \in
    \bN}$ is effective, by the equivalent statements in 
  Remark~\ref{rem:effective}.
  Hence the inverse of this arrow is the morphism 
  $\tau_{\geq -n}(\gamma)$ which is therefore an
  isomorphism. Since $n \in \bN$ was arbitrary, this implies that
  $\gamma$ is an isomorphism. The claim follows.
\end{proof}

We now provide instances of inverse and direct truncation
systems that are effective. 
The following proposition is essentially Remark~2.3 in
\cite{neeman-homotopy-limits}.

\begin{proposition}
\label{p:ab4-trunc}
  Let $\mathcal{A}$ be an abelian category. If $\mathcal{A}$ has
  countable products which are exact
  then any inverse truncation system in $\D(\mathcal{A})$ is
  effective.
  Dually, if $\mathcal{A}$ has exact countable coproducts
  then any direct truncation system in $\D(\mathcal{A})$ is
  effective.
\end{proposition}

\begin{proof}
  Since the statements are dual,
  it is enough to prove one of them.
  We prove the first statement.

  It is clear that $\D(\mathcal{A})$ has all countable products.
  Let $(F_n)_{n \in \mathbb{N}}$ be an inverse
  truncation system in $\D(\mathcal{A})$.
  Consider a homotopy limit $\holim_n F_n$ together with a
  defining triangle  
  \begin{equation*}
    \holim_n F_n \to \prod_n F_n \xrightarrow{1 - \sigma}
    \prod_n F_n \to \Sigma \holim_n F_n,
  \end{equation*}
  where $\sigma$ is the shift map.
  Since $\mathcal{A}$ has exact countable products,
  the functor $\HH^p(-)$ preserves countable products for all $p$.
  Hence, we get an exact sequence
  \begin{equation*}
    \dots
    \xrightarrow{\delta} 
    \HH^p(\holim_n F_n) \to 
    \prod_n \HH^p(F_n) \xrightarrow{1 - \sigma} 
    \prod_n \HH^p(F_n) \xrightarrow{\delta} 
    \HH^{p + 1}(\holim_n F_n)
    \ra \dots
  \end{equation*}
  in $\mathcal{A}$ where $\sigma$ is again the shift  map.
  Note that the inverse system $(\HH^p(F_n))_{n \in \mathbb{N}}$
  is zero for $n < -p$ and constant for $n \geq -p$ (cf.\
  Remark~\ref{rem:ITS-alternative}). 
  Hence Remark~\ref{rem:1-shift-split-epi} below shows that 
  all morphisms $1- \sigma$ in our exact sequence are
  epimorphisms, i.e., all connecting 
  morphisms $\delta$ 
  vanish. Hence the canonical morphism
  $$
  \HH^p(\holim_n F_n) \to \lim_n \HH^p(F_n)
  $$
  is an isomorphism, i.e., our inverse truncation system is effective.
\end{proof}

\begin{remark}
  \label{rem:1-shift-split-epi}
  If an inverse system $(G_n)_{n \in \bN}$ in an additive
  category with countable
  products is \emph{eventually constant} in the sense that
  there is an $N \in \bN$ 
  such that all
  transition morphisms $G_{n+1} \ra G_n$ for $n \geq N$ are
  isomorphisms, then the morphism ``identity minus shift''
  \begin{equation*}
    1-\sigma \colon \prod_{n \in \bN} G_n \ra \prod_{n \in \bN} G_n
  \end{equation*}
  is a split epimorphism: there is a morphism $f$ in the other
  direction such that $(1-\sigma) \circ f = \id$. We leave the
  easy proof of this fact to the reader.
\end{remark}

\begin{example}
  \label{ex:modules-effective}
  If $R$ is an arbitrary ring, then all 
  inverse and direct truncation systems in the derived category
  $\D(\Mod(R))$ of $R$-modules are effective, by 
  Proposition~\ref{p:ab4-trunc}.
\end{example}

\begin{example}
  \label{ex:ringed-topos-effective}
  If $(X, \mathcal{O})$ is a ringed topos,
  then the category of $\mathcal{O}$-modules is Grothendieck
  abelian. 
  In particular, it has arbitrary small products and coproducts and
  all coproducts are exact;
  in particular, any direct truncation system in $\D(X)$ is
  effective, by Proposition~\ref{p:ab4-trunc}. 
  Products, however, need not be exact.
  Nevertheless, certain inverse truncation systems may still be
  effective, see Proposition~\ref{p:Dqc-truncation-complete} below.
\end{example}


The following proposition follows from the proof of
\cite[\sptag{0D6M}]{stacks-project},
which basically is an abstract version of a result by Bökstedt--Neeman~\cite[Lemma~5.3]{neeman-homotopy-limits}.

\begin{proposition}
  \label{p:Dqc-truncation-complete}
  Let $X$ be an algebraic stack.
  Then any inverse truncation system in $\D(X)$ with terms in
  $\D_\qc(X)$ is effective. 
\end{proposition}

\begin{proof}
  It is clear that $\D(X)$ has all products.
  Let $(F_n)_{n \in \mathbb{N}}$
  be an inverse truncation system 
  in $\D(X)$ whose terms $F_n$ lie in $\D_\qc(X)$.

  By the definition of an inverse truncation system, the 
  transition morphism
  $F_{n+1} \ra F_n$ factors as 
  $F_{n+1} \ra \tau_{\geq -n}F_{n+1} \sira F_n$, 
  and we obtain a triangle
  \begin{equation*}
    \Sigma^{n + 1}\HH^{-n-1}(F_{n+1})
    \to F_{n+1}
    \to F_n
    \to \Sigma^{n + 2}\HH^{-n-1}(F_{n+1})
  \end{equation*}
  for each $n \in \mathbb{N}$.
  Fix $p \in \bZ$. 
  Since the cohomology sheaves are assumed to be quasi-coherent,
  the functor $\HH^p(U, -)$ vanishes on the first and fourth term of
  this triangle 
  for 
  all $n \geq -p$ and all affine schemes $U$ which are smooth
  over $X$. 
  Hence $\HH^p(U, F_{n+1}) \ra \HH^p(U, F_n)$ is an isomorphism
  for all $n \geq -p$, i.e., 
  the inverse system 
  $(\HH^p(U, F_n))_{n \in \bN}$
  of abelian groups
  is eventually constant.
  By Remark~\ref{rem:1-shift-split-epi}, 
  the morphism ``identity minus shift''
  \begin{equation*}
    1-\sigma \colon \prod_n\HH^p(U,F_n)
    \ra \prod_n \HH^p(U,F_n)
  \end{equation*}
  is a (split) epimorphism. Its kernel is $\lim_n \HH^p(U,F_n)$.

  Note that the functors $\HH^q(U, -) = \HH^q(\Rd \Gamma(U, -|_U))$
  commute with 
  products since the functors restriction to $U$ and $\Rd
  \Gamma(U, -)$ are right 
  adjoints and products 
  are exact in the category of abelian groups.
  Hence applying $\HH^p(U, -)$
  to the defining triangle of the homotopy limit $\holim_n F_n$ 
  yields an exact sequence
  \begin{equation*}
    \dots
    \ra
    \HH^p(U, \holim_n F_n) \to 
    \prod_n \HH^p(U, F_n) \xrightarrow{1 - \sigma} 
    \prod_n \HH^p(U, F_n)
    \ra
    \dots
  \end{equation*}
  of abelian groups. Since all morphisms $1-\sigma$ are
  surjective, as observed above, this sequence splits into 
  short exact sequences, and
  we get isomorphisms
  \begin{equation*}
    \HH^p(U,\holim_n F_n) 
    \xrightarrow{\sim}
    \lim_n \HH^p(U,F_n)
    \xrightarrow{\sim}
    \HH^p(U,F_m)
  \end{equation*}
  for all $m \geq -p$.
  For an arbitrary object $G$ in $\D(X)$,
  the sheaf $\HH^p(G)$ is 
  the sheaf associated to the presheaf
  $U \mapsto \HH^p(U, G)$.
  Hence the canonical map $\HH^p(\holim_n F_n) \to
  \HH^p(F_m)$ 
  is an isomorphism for all $m \geq -p$.
  This means that our system is effective, 
  by Remark~\ref{rem:effective}.
\end{proof}

\subsection{Decompositions of unbounded derived categories}
\label{sec:decomp-unbo-deriv}

Now we can state and prove Theorem~\ref{t:sod-DBA}.

\begin{remark}
\label{rem:cohomology-split}
Let $\mathcal{B}$ be a weak Serre subcategory of an abelian
category $\mathcal{A}$
and let $\mathcal{T}$ and $\mathcal{F}$ be abelian subcategories of
$\mathcal{B}$ forming a torsion pair $(\mathcal{T}, \mathcal{F})$ in
$\mathcal{B}$.
Consider an arbitrary triangle
\begin{equation*}
  B' \to B \to B'' \to \Sigma B'
\end{equation*}
in $\D_\mathcal{B}(\mathcal{A})$.
Then we have $B'$ in $\D_\mathcal{T}(\mathcal{A})$
and $B''$ in $\D_\mathcal{F}(\mathcal{A})$
if and only if the long exact cohomology sequence splits up
into short exact sequences
\begin{equation*}
  0 \to \HH^p(B')
  \to \HH^p(B)
  \to \HH^p(B'')
  \to 0
\end{equation*}
with first term $\HH^p(B')$ in $\mathcal{T}$ and third term
$\HH^p(B'')$ in
$\mathcal{F}$, for each $p \in \bZ$. The ``if''-part is trivial,
and the ``only if''-part
is a simple consequence of the fact that $\Hom_\mathcal{B}(F, T)$
vanishes for each object $F$ in $\mathcal{F}$ and $T$
in $\mathcal{T}$, by 
Proposition~\ref{p:abelian-torsion-pair}.\ref{enum:Hom-FT}.
\end{remark}

\begin{theorem}
  \label{t:sod-DBA}
  Let $\mathcal{B}$ be a weak Serre subcategory of an abelian
  category $\mathcal{A}$.
  Let $(\mathcal{T}, \mathcal{F})$ be a torsion pair in
  $\mathcal{B}$ where both $\mathcal{T}$ and
  $\mathcal{F}$ are abelian subcategories of $\mathcal{B}$.
  Assume that 
  \begin{equation*}
    \Ext_\mathcal{A}^n(T, F)=0
  \end{equation*}
  for all objects $T \in \mathcal{T}$, $F \in \mathcal{F}$ and
  all integers $n \in \bZ$. 
  Assume that $\D(\mathcal{A})$ has countable products and
  coproducts. 
  If every inverse and every direct truncation
  system in 
  $\D(\mathcal{A})$ with terms in $\D_\mathcal{B}(\mathcal{A})$
  is effective, 
  in the sense of Definition~\ref{d:truncation-system},
  then 
  \begin{equation*}
    \D_{\mathcal{B}}(\mathcal{A})=
    \langle \D_{\mathcal{F}}(\mathcal{A}), 
    \D_{\mathcal{T}}(\mathcal{A})\rangle
  \end{equation*}
  is a semi-orthogonal decomposition.
\end{theorem}

\begin{proof}
\textbf{Semi-orthogonality}:
We claim that
\begin{equation}
  \label{eq:Hom-vanishing-unbounded}
  \Hom_{\D(\mathcal{A})}(L, R)=0
\end{equation}
for all objects $L \in \D_\mathcal{T}(\mathcal{A})$ and $R \in
\D_\mathcal{F}(\mathcal{A})$.

Our proof is a straightforward d\'evissage argument.
Assume first that $L$ is in $\D^-_\mathcal{T}(\mathcal{A})$ and that $R$ is
in $\D^+_\mathcal{F}(\mathcal{A})$.
Then any morphism $L \to R$ factors as $L \to \tau_{\geq n} L \to
\tau_{\leq m} R \to R$ for sufficiently small $n$ and
sufficiently large $m$.
Hence the morphism vanishes by Theorem~\ref{t:sod-DbBA}.

Keep the assumption on $L$ and assume that $R$ is an arbitrary
object of $\D_\mathcal{F}(\mathcal{A})$.
Effectiveness of inverse
truncation systems with terms in 
$\D_\mathcal{B}(\mathcal{A})$ and
Lemma~\ref{l:holim-truncations} provide a triangle
\begin{equation*}
  R \ra 
  \prod_{n \in \bN} \tau_{\geq -n} R 
  \ra
  \prod_{n \in \bN} \tau_{\geq -n} R 
  \ra
  \Sigma R
\end{equation*}
in $\D(\mathcal{A})$.
The functor $\Hom_{\D(\mathcal{A})}(L,-)$ vanishes on the second and
third term and their shifts by the universal property of
products, boundedness from below of
the truncations and the vanishing result already proven.
Hence it vanishes also on the first term as desired.

Similarly, 
effectiveness of direct
truncation systems with terms in 
$\D_\mathcal{B}(\mathcal{A})$
and the dual of
Lemma~\ref{l:holim-truncations}
give
\eqref{eq:Hom-vanishing-unbounded}
for arbitrary objects $L$ in $\D_\mathcal{T}(\mathcal{A})$ and
$R$ in $\D_\mathcal{F}(\mathcal{A})$.

\textbf{Completeness:}
To finish the proof, we need to show that any object $B$ in
$\D_{\mathcal{B}}(\mathcal{A})$ fits into a triangle
\begin{equation}
\label{eq:tria-complete}
T \to B \to F \to \Sigma T 
\end{equation}
with $T$ in $\D_{\mathcal{T}}(\mathcal{A})$ and
$F$ in $\D_{\mathcal{F}}(\mathcal{A})$.

Assume first that $B$ is an object in $\D^-_\mathcal{B}(\mathcal{A})$
and consider its associated standard inverse truncation system
$(B_n)_{n \in \mathbb{N}}$ where $B_n = \tau_{\geq -n} B$
(see Example~\ref{ex:ITS}).
Since the cohomology of each $B_n$ is bounded,
the semi-orthogonal decomposition in Theorem~\ref{t:sod-DbBA}
provides commutative squares
\begin{equation*}
  \xymatrix{
    T_n \ar[r]\ar[d] & T_{n+1}\ar[d] \\
    B_n \ar[r] & {B_{n+1},} 
  }
  \qquad
  \xymatrix{
    \HH^p(T_n) \ar[r]\ar[d] & \HH^p(T_{n+1})\ar[d] \\
    \HH^p(B_n) \ar[r] & \HH^p(B_{n+1}), 
  }
\end{equation*}
where each $T_n$ is the projection of $B_n$ to
$\D^\bd_\mathcal{T}(\mathcal{A})$ along this semi-orthogonal decomposition.
By Remark~\ref{rem:cohomology-split}, the vertical maps in the right
diagram are just the torsion subobjects with respect to the
torsion pair $(\mathcal{T}, \mathcal{F})$.
The characterization of inverse truncation systems in
Remark~\ref{rem:ITS-alternative} 
and the fact that 
$(B_n)_{n \in \mathbb{N}}$
is such a system make it 
clear that $(T_n)_{n \in \mathbb{N}}$
is such a system as well.
By assumption, it is effective, and hence obtained from its
homotopy limit $T$ in $\D(\mathcal{A})$ by truncation in a
canonical way, see
Remark~\ref{rem:effective}. 
Without loss of generality we may hence assume that 
$T_n=\tau_{\geq -n}T$ for all $n \in \bN$.
It is clear that $T \in \D^-_\mathcal{T}(\mathcal{A})$.
Note that $T$ and $B$ are the homotopy limits of 
the effective inverse truncation systems 
$(T_n)_{n \in \mathbb{N}}$ and $(B_n)_{n \in \mathbb{N}}$,
by Lemma~\ref{l:holim-truncations}.
Hence the morphism $(T_n)_{n \in \mathbb{N}} \ra (B_n)_{n \in \mathbb{N}}$
of inverse systems induces a morphism
$T \ra B$ which makes the following two squares commutative.
\begin{equation*}
  \xymatrix{
    {T} \ar[r]\ar[d] & {T_{n}} \ar[d] \\
    {B} \ar[r] & {B_{n}} 
  }
  \qquad
  \xymatrix{
    {\HH^p(T)} \ar[r]\ar[d] & {\HH^p(T_{n})}\ar[d] \\
    {\HH^p(B)} \ar[r] & {\HH^p(B_{n})} 
  }
\end{equation*}
Obviously, the horizontal arrows in the right
square are isomorphisms for all $n \geq -p$.
Hence $\HH^p(T) \to \HH^p(B)$ is the torsion
subobject with respect to the torsion pair $(\mathcal{T},
\mathcal{F})$, for each $p \in \bZ$.
Complete the morphism $T \to B$ to a triangle as in \eqref{eq:tria-complete}.
Then the associated long exact sequence on cohomology 
splits up into short exact sequences with $\HH^p(F)$ in $\mathcal{F}$.
In particular, the object $F$ is in $\D^-_\mathcal{F}(\mathcal{A})$.

This establishes the semi-orthogonal decomposition
$
\D^-_{\mathcal{B}}(\mathcal{A}) =
\langle \D^-_{\mathcal{F}}(\mathcal{A}), 
\D^-_{\mathcal{T}}(\mathcal{A})\rangle.
$
By the dual version of the argument above,
starting with an object $B$ in $\D_\mathcal{B}(\mathcal{A})$
and 
using the fact that all direct truncation systems with terms in
$\D_{\mathcal{B}}(\mathcal{A})$ are effective,
we get the desired semi-orthogonal decomposition of
$\D_\mathcal{B}(\mathcal{A})$. 
\end{proof}


\section{Gerbes}
\label{sec:gerbes-and-twisted}
In this section, we collect some results from the theory on gerbes,
bandings by abelian groups and the relation to cohomology.
Throughout the section, $S$ will be an arbitrary algebraic stack.
By a \emph{stack} we mean a, not necessarily algebraic, stack in groupoids over the
big \emph{fppf} site of schemes over $S$. 
The symbol $\Delta$ will always denote an abelian group
in~$S_\fppf$. 
Later on, we will specialize to the situation when $\Delta$ is a diagonalizable group.
Accordingly, we will use multiplicative notation for
$\Delta$ and its cohomology groups $\HH^i(S_\fppf, \Delta)$.

\begin{definition}[{\cite[Definition~3.15]{lmb2000}, \cite[\sptag{06NZ}]{stacks-project}}]
A \emph{gerbe} over $S$ is a stack (in groupoids) $\alpha\colon X \to S$ satisfying the following
conditions.
\begin{enumerate}
\item
The diagonal $X \to X\times_S X$ is an epimorphism (see \cite[Definition~3.6]{lmb2000}).
\item
\label{def-gerbe-structure}
The structure morphism $\alpha\colon X \to S$ is an epimorphism.
\end{enumerate}
A gerbe is called \emph{trivial} if the structure morphism in \ref{def-gerbe-structure} splits.
\end{definition}

The prototypical example of a gerbe is the \emph{classifying stack} $\BB G$
for a group $G$ in $S_\fppf$.
In fact, any gerbe can be viewed as an \emph{fppf}-form of a classifying stack.

For any stack $\alpha\colon X \to S$,
we may consider its \emph{inertia stack} $I_{X/S} \to X$ (\cite[\sptag{034I}]{stacks-project}),
which is a group object in $X_\fppf$.
Its points over a morphism $x \colon U \to X$ are simply the 2-automorphisms $\gamma$ of $x$
mapping to the identity under the composition $\alpha \circ x$.
Any object $\mathcal{F}$ in $X_\fppf$ is endowed with a canonical right action by the inertia $I_{X/S}$,
called the \emph{inertial action}.
Explicitly, it is given on sections over $x$ by
\begin{equation}
\label{eq-inertial-action}
\mathcal{F}(x)\times I_{X/S}(x) \to \mathcal{F}(x),\qquad
(s, \gamma) \mapsto \mathcal{F}(\gamma)(s).
\end{equation}
We recall the following fundamental fact about the inertial action.
\begin{proposition}
\label{p-gerbe-pb-ff}
Let $\alpha\colon X \to S$ be a gerbe.
Then the functor $\alpha^{*}\colon S_\fppf \to X_\fppf$ is fully faithful with essential
image consisting of sheaves on which the inertia $I_{X/S}$ acts trivially.
\end{proposition}
\begin{proof}
See for instance \cite[Lemma~2.1.1.17]{lieblich2004}.
\end{proof}

\begin{definition}
A gerbe $\alpha\colon X \to S$ is called \emph{abelian} if the inertia group $I_{X/S} \to X$
is abelian.
We denote the full subcategory of stacks over $S$ which are abelian gerbes by
$\AbGerbe(S)$.
\end{definition}

The inertia acts on itself by conjugation.
In particular, Proposition~\ref{p-gerbe-pb-ff} shows that the inertia group of an abelian gerbe $\alpha\colon X \to S$
descends to an abelian
group object in $S_\fppf$, namely $\alpha_* I_{X/S}$.
Furthermore, given any 1-morphism $\rho\colon X \to Y$ of stacks over $S$,
we get an induced group homomorphism $I_{X/S} \to \rho^*I_{Y/S}$ in $X_\fppf$,
and if $X$ and $Y$ are abelian gerbes, this descends to a group homomorphism in $S_\fppf$.
This association gives a functor
\begin{equation}
\label{eq-band-functor}
\Band\colon \AbGerbe(S) \to \AbGroup(S)
\end{equation}
to the category of abelian group objects in $S_\fppf$.

\begin{definition}
\label{def-banding}
Let $\Delta$ be an abelian group object in $S_\fppf$.
The essential fiber of the functor \eqref{eq-band-functor} over $\Delta$ is called
the \emph{category of gerbes banded by $\Delta$} or \emph{of $\Delta$-gerbes}. 
Explicitly, a $\Delta$-gerbe is given by a pair $(\alpha, \iota)$ where $\alpha\colon X \to S$ is an abelian
gerbe and $\iota\colon \Delta \xrightarrow{\sim} \alpha_*I_{X/S}$ is a group isomorphism
called a \emph{banding} of $\alpha$ by $\Delta$.
The group $\Delta$ is called the band of $(\alpha,\iota)$.
A morphism of $\Delta$-gerbes is simply a morphism of stacks such that corresponding morphism on inertia
induces the identity on $\Delta$ via the bandings.
By abuse of notation, we usually denote a $\Delta$-gerbe $(\alpha, \iota)$ by $\alpha$ or even $X$.
\end{definition}

\begin{remark}
Giraud introduces a more general notion of banding applying to non-abelian gerbes
\cite[Chapter~IV]{giraud1971}.
We will not need this, more complicated, theory here.
\end{remark}

Note that a 1-morphism of gerbes over $S$ is always an epimorphism (\cite[Lemma~3.17]{lmb2000}),
and that it is an isomorphism if and only if it induces an isomorphism on inertia.
In particular, the category of $\Delta$-gerbes is a 2-groupoid.

\begin{example}
\label{ex:simple-gerbes}
Let $S$ be an algebraic stack,
and let $\Delta$ and $\Delta'$ be abelian groups in $S_\fppf$.
\begin{enumerate}
\item
The classifying stack $\BB \Delta$ is endowed with a canonical structure of $\Delta$-gerbe.
\item
Given a pair of abelian gerbes $X$ and $X'$ over $S$,
banded by $\Delta$ and $\Delta'$, respectively,
the product $X\times_S X'$ has a naturally defined banding by $\Delta\times \Delta'$.
\end{enumerate}
\end{example}

The following example illustrates that it is very natural to consider gerbes over
genuine algebraic stacks.
\begin{example}
\label{ex-gerbes}
Let $S$ be an algebraic stack. Assume that we have a central extension
$1 \to \Delta \to G \to H \to 1$ of groups in $S_\fppf$.
Then the induced morphism of classifying stacks $\BB G \to \BB H$
has a canonical structure of a $\Delta$-gerbe.
The following examples are two important special cases with $S = \Spec \mathbb{Z}$.
\begin{enumerate}
\item
\label{it-universal-projective}
The morphism $G \to H$ is the quotient map $\GL_n \to \PGL_n$.
In this case we obtain a $\GGm$-gerbe $\BB\GL_n \to
\BB\PGL_n$.
\item
\label{it-universal-root}
The morphism $G \to H$ is the morphism $\GGm \to \GGm$ given by
$x \mapsto x^n$. 
In this case the corresponding $\mu_n$-gerbe $\BB \GGm \to \BB \GGm$ can be thought
of as the $n$-th root stack of the universal line bundle $[\bA_1/\GGm]$ on $\BB
\GGm$ (cf.~\cite[Definition~2.2.6]{cadman2007}).
\end{enumerate}
\end{example}

Next we discuss functoriality of banded gerbes with respect to change of group.

\begin{construction}
\label{cons-gerbe-pf}
Let $\varphi\colon \Delta \to \Delta'$ be a homomorphism of abelian group objects in $S$,
and let $\alpha\colon X \to S$ be a gerbe banded by $\Delta$.
We construct a gerbe $\varphi_*\alpha\colon \varphi_*  X \to S$
banded by $\Delta'$ together with a morphism
\begin{equation}
\label{eq-gerbe-pf}
\rho \colon X \to \varphi_*X
\end{equation}
of gerbes over $S$ inducing the homomorphism $\varphi$ on the bands as follows.
 
Consider $X \times \BB \Delta'$ with its natural banding by $\Delta \times \Delta'$.
We define $\varphi_*X$ as the \emph{rigidification} (see \cite[Appendix~A]{aov2008})
of $X \times \BB \Delta'$ in the kernel of the epimorphism
$\Delta \times \Delta' \to \Delta'$ given by $(a, b) \mapsto \varphi(a) b$.
The morphism $\rho$ is the composition of the obvious map $X \to X\times \BB \Delta'$ followed by the
rigidification map.
\end{construction}

The construction above shows that given a gerbe banded by $\Delta$,
any morphism $\varphi\colon \Delta \to \Delta'$ lifts to a morphism of banded gerbes.
The following proposition,
which we surprisingly could not find in the literature,
shows that this lift is essentially unique.
In particular, Construction~\ref{cons-gerbe-pf} is functorial in a weak sense.

\begin{proposition}
\label{gerbe-functorial-group}
Let $\Delta$ be an abelian group in $S_\fppf$,
and let $\alpha\colon X \to S$ be a gerbe banded by $\Delta$.
Then the obvious functor
\begin{equation}
\Band_{X/} \colon \AbGerbe(S)_{X/} \to \AbGroup(S)_{\Delta/}
\end{equation}
from the category of abelian gerbes over $S$ under $X$ to the category of abelian groups in $S_\fppf$
under $\Delta$ is an equivalence of categories.
\end{proposition}
\begin{proof}
Construction~\ref{cons-gerbe-pf} shows that the functor
$\Band_{X/}$ is essentially surjective on objects.
Let $\rho\colon X \ra Y$ and  $\rho'\colon X \ra Y'$ be objects of $\AbGerbe(S)_{X/}$,
and denote their images in $\AbGroup(S)_{\Delta/}$ by
$\varphi\colon \Delta \to \Gamma$ and~$\varphi'\colon \Delta \to \Gamma'$, respectively.
We need to show that
\begin{equation}
\label{eq-band-equivalence}
\Hom_{X/}(\rho, \rho') \to \Hom_{\Delta/}(\varphi, \varphi')
\end{equation}
is and equivalence of categories.
To do so, we first note that \eqref{eq-band-equivalence} is functorial in~$S$,
and that we in fact have a morphism of stacks
\begin{equation}
\label{eq-band-stacky-equivalence}
\Phi\colon \sheafHom_{X/}(\rho, \rho') \to \sheafHom_{\Delta/}(\varphi, \varphi')
\end{equation}
over~$S$.
Explicitly, this morphism is constructed as the canonical map in the diagram
\begin{equation}
\label{eq-big-and-scary}
\xymatrix{
\sheafHom_{\Delta/}(\varphi, \varphi')\ar[ddd]\ar[rrr] & & & S \ar[ddd]^{\varphi'}\\
& \sheafHom_{X/}(\rho, \rho')\ar[d]\ar[r]\ar[ul]_\Phi & S\ar[ur]^\id \ar[d]^{\rho'} & \\
& \sheafHom(Y, Y') \ar[r]^{-\circ\rho}\ar[dl] & \sheafHom(X, Y')\ar[dr] & \\
\sheafHom(\Gamma, \Gamma') \ar[rrr]^{-\circ\varphi} & & & \sheafHom(\Delta, \Gamma'), \\
}
\end{equation}
where the inner and the outer squares are cartesian.

To verify that $\Phi$ is an isomorphism, we may work locally on~$S$.
In particular, we may assume that the stacks $X, Y$ and~$Y'$ are the classifying
stacks
$\BB \Delta$, $\BB \Gamma$ and~$\BB \Gamma'$, respectively.
Furthermore, we may assume that the morphisms $\rho$ and $\rho'$ are given by extension of torsors along
$\varphi$ and $\varphi'$, respectively.
In this situation, we have $\sheafHom(X, Y') \cong [\sheafHom(\Delta, \Gamma')/\Gamma']$,
where $\Gamma'$ acts trivially,
and the map to $\sheafHom(\Delta, \Gamma')$ is induced by forgetting the group action
(see e.g. Abramovich~\emph{et~al.}~\cite[Lemma~3.9(1)]{aov2008} and its proof).
Plugging this into \eqref{eq-big-and-scary},
together with the similar description for $\sheafHom(Y, Y')$ and the obvious descriptions for $\rho'$
and $-\circ\rho$,
it becomes clear that $\Phi$ is an isomorphism.
\end{proof}

\begin{construction}
\label{cons-gerbe-operations}
Let $\Delta$ be an abelian group in $S_\fppf$,
and let $\alpha, \beta\colon X \to S$ be $\Delta$-gerbes.
Denote the multiplication map on $\Delta$ by $m\colon \Delta \times \Delta \to \Delta$,
and the $n$-th power map, for $n\in \mathbb{Z}$, by~$p_n\colon \Delta \to \Delta$.
Using Construction~\ref{cons-gerbe-pf}, we define
\begin{equation}
\label{eq-gerbe-mult-inverse}
\alpha\beta := m_*(\alpha \times \beta), \qquad
\alpha^n := (p_n)_*\alpha.
\end{equation}
\end{construction}
Morally, Construction~\ref{cons-gerbe-operations} gives the 2-groupoid of $\Delta$-gerbes
the structure of an abelian group,
a statement which presumably could be made precise by use of Proposition~\ref{gerbe-functorial-group}.
At least, it is clear that the set of isomorphism classes of $\Delta$-gerbes forms an abelian group.

It is easy to see that we also have functoriality with respect to change of stacks.
That is, let $f\colon T \to S$ be a morphism of algebraic stacks
and let $\alpha\colon X \to S$ be a $\Delta$-gerbe.
Then the pull-back $f^*\alpha$ has a natural structure of $f^*\Delta$-gerbe.

\begin{theorem}[Giraud]
\label{t:gerbes-H2}
Let $S$ be an algebraic stack and let $\Delta$ be an abelian group in~$S_\fppf$.
Then the group $\HH^2(S_\fppf, \Delta)$ is canonically isomorphic to the set of isomorphism
classes of $\Delta$-gerbes over $S_\fppf$,
with its group structure induced by the operations in Construction~\ref{cons-gerbe-operations}.
Given a $\Delta$-gerbe $\alpha$ over $S_\fppf$,
we denote its class in $\HH^2(S_\fppf, \Delta)$ by~$[\alpha]$.
We have the following functorial properties:
\begin{enumerate}
\item
Given a morphism $f\colon T \to S$ of algebraic stacks,
we have $[f^*\alpha] = f^*[\alpha]$ in $H^2(T_\fppf, f^*\Delta)$.
\item
Given a homomorphism of abelian groups $\varphi\colon \Delta \to \Delta'$,
we have $[\varphi_\ast \alpha] = \varphi_*[\alpha]$ in $H^2(S_\fppf, \Delta')$.
\end{enumerate}
\end{theorem}
\begin{proof}
This is worked out in various places in \cite{giraud1971}.
Functoriality with respect to change of groups is described in Section~IV.3.1.
Comparison with usual cohomology is in Section~IV.3.4,
and functoriality with respect to change of topos in Section~V.1.
\end{proof}

\section{Twisted sheaves}
\label{sec:quasi-coher-sheav}
In this section, $\Delta$ will denote a diagonalizable group scheme (over $\Spec \mathbb{Z}$)
with character group~$A := \Hom_{\Spec\mathbb{Z}}(\Delta, \GGm)$.
Recall that $A$ is a finitely generated abelian group.
We will usually use multiplicative notation for $A$,
except when we identify it with something like the additive group of $\mathbb{Z}$.

Let $S$ be an algebraic stack.
By a $\Delta$-gerbe on $S$ we mean a $\Delta_S$-gerbe on
$S$.
Fix a $\Delta$-gerbe $\alpha\colon X \to S$.
Note that $X$ is isomorphic to $\BB \Delta_S$ locally on $S$.
In particular, $X$ is an algebraic stack,
and it is reasonable to talk about quasi-coherent sheaves on~$X$.

\begin{remark}
\label{rem-gerbe-ff}
Let $\alpha\colon X \to S$ be a $\Delta$-gerbe.
Note that the structure morphism $\alpha$ is smooth and surjective.
Indeed, this is true for any algebraic gerbe
\cite[Proposition~A.2]{bergh2017}.
Moreover, assume that $Y$ is a $\Delta'$-gerbe over $S$
and let $\rho\colon X \to Y$ be a morphism over $S$ with induced morphism $\varphi\colon \Delta_S \to \Delta'_S$ on bands.
Then $\rho$ factors as $X \to X\times\BB \Delta' \to \varphi_*X \cong Y$ by Construction~\ref{cons-gerbe-pf}
combined with Proposition~\ref{gerbe-functorial-group}.
Note that the first morphism in the factorization is a
$\Delta'$-torsor 
and that the second morphism is a gerbe banded by~$\Delta$.
In particular, the morphism $\rho$ is faithfully flat.
See also Remark~\ref{rem-gerbe-concentrated} for cohomological properties of $\rho$.
\end{remark}

\begin{definition}
\label{def-homogeneous-sheaf}
Let $\mathcal{F}$ be an object in $\Mod(X_\fppf, \mathcal{O}_X)$,
and let $\chi\colon \Delta_S \to \GGms{S}$ be a character.
Note that $\Delta_X$ acts on $\mathcal{F}$ from the right via the inertial action \eqref{eq-inertial-action}
and the banding (see Definition~\ref{def-banding}) of $X$,
and that $\GGm$ acts on $\mathcal{F}$ from the left via the inclusion $\GGms{X} \subset \mathcal{O}_X$.
We define the \emph{$\chi$-homogeneous subsheaf} $\mathcal{F}_\chi$ of $\mathcal{F}$ by
\begin{equation}
  \label{eq:chi-subsheaf}
  \mathcal{F}_\chi = \{s \in \mathcal{F} \mid 
  \text{$s\cdot\gamma = \chi(\gamma)s$ for all $\gamma
  \in \Delta_X$}\}.
\end{equation}
An object $\mathcal{F}$ of $\Mod(X_\fppf, \mathcal{O}_X)$ is
called \emph{$\chi$-homogeneous}
or \emph{homogeneous of degree $\chi$}
provided that $\mathcal{F} = \mathcal{F}_\chi$,
and we denote the full subcategory of $\chi$-homogeneous objects in $\Qcoh(X_\fppf, \mathcal{O}_X)$
by $\Qcoh_\chi(X_\fppf, \mathcal{O}_X)$.
\end{definition}

\begin{remark}
\label{rem:def-Qcoh-chi}
Note that the expression
\eqref{eq:chi-subsheaf} does not make sense in the lisse--\'etale
topos unless $\Delta_X$ is smooth over $X$.
When we talk about homogeneous sheaves in $\Qcoh(X_\liset, \mathcal{O}_X)$,
we implicitly transport the subcategory
$\Qcoh_\chi(X_\fppf, \mathcal{O}_X) \subset \Qcoh(X_\fppf, \mathcal{O}_X)$
to a subcategory $\Qcoh_\chi(X_\liset, \mathcal{O}_X) \subset \Qcoh(X_\liset, \mathcal{O}_X)$
under the equivalence obtained by restriction.
We will usually simply write $\Qcoh_\chi(X)$ for any of these subcategories. 
\end{remark}

\begin{definition}[Twisted sheaves]
\label{def:twisted}
Let $\alpha\colon X \to S$ be a $\GGm$-gerbe and let $\id\colon \GGm \to \GGm$
be the identity character.
An \emph{$\alpha$-twisted quasi-coherent sheaf on $S$} is simply an object of $\Qcoh_\id(X)$.
We will often use the notation $\Qcoh_{\alpha}(S)$ instead
of $\Qcoh_\id(X)$.
\end{definition}

\begin{remark}
\label{rem-twisted-abuse}
Since the category $\Qcoh_{\alpha}(S)$, up to equivalence,
obviously depends only on the class of $\alpha$ in $\HH^2(S_\fppf, \GGm)$,
we will sometimes abuse notation and write $\Qcoh_{\alpha}(S)$
even if $\alpha$ is just a cohomology class in $\HH^2(S_\fppf, \GGm)$.
\end{remark}

\begin{remark}
\label{rem-twisted-2-cocycle}
Twisted sheaves may also be described using 2-cocycles.
This approach was taken by C{\u{a}}ld{\u{a}}raru in \cite[Definition~1.2.1]{caldararu2000}.
That his definition is essentially equivalent to the definition above was proved by
Lieblich in~\cite[Proposition~2.1.3.11]{lieblich2004}. 
\end{remark}

We now give a precise formulation of the well-known fact that
any quasi-coherent sheaf on a suitable gerbe
decomposes into its homogeneous subsheaves (c.f.~\cite[Proposition~2.2.1.6]{lieblich2004}).

\begin{theorem}
  \label{t-qc-gerbe-split}
  Let $S$ be an algebraic stack and let $\alpha\colon X \to S$ be a
  gerbe 
  banded by a diagonalizable group $\Delta$.
  Then $\Qcoh_\chi(X)$ is
  a Serre subcategory (see~\cite[\sptag{02MO}]{stacks-project}) of $\Qcoh(X)$
  for any character $\chi$ of $\Delta_S$,
  and taking the coproduct gives an equivalence
  \begin{equation}
    \label{eq-eq-qc-gerbe}
    \prod_{\chi \in A}\Qcoh_\chi(X) \sira \Qcoh(X),
    \qquad
    (\mathcal{F}_\chi)_\chi \mapsto \bigoplus_{\chi \in A}
    \mathcal{F}_\chi, 
  \end{equation}
  of abelian categories, where $A$ denotes the character group of $\Delta$.
  
  Furthermore, assume that $\mathcal{F}_\chi$ and $\mathcal{G}_\psi$ are quasi-coherent
  $\mathcal{O}_X$-modules which are homogeneous for $\Delta_S$-characters $\chi$ and $\psi$, respectively.
  Then the quasi-coherent $\mathcal{O}_X$-modules
  \begin{equation}
    \label{eq-qc-monoidal}
    \mathcal{F}_\chi \otimes \mathcal{G}_\psi, \qquad
    \sheafHom(\mathcal{G}_\psi, \mathcal{F}_\chi)
  \end{equation}
  are $\chi\psi$- and $\chi\psi^{-1}$-homogeneous,
  respectively. 
\end{theorem}

\begin{proof}
Throughout the proof,
we work with sheaves in $X_\fppf$,
which we view as the topos of sheaves on the site of affine schemes over $X$.
Fix a quasi-coherent sheaf $\mathcal{F}$ of $\mathcal{O}_X$-modules.

  Let $x$ be an object of $X$ lying over an affine scheme $\Spec A$.
  Then $\mathcal{F}(x)$ is an $A$-module with an $A$-linear right action of the group 
  $\Delta(A):=\Delta(\Spec A)$ coming from the inertial action via the banding.
  Given a morphism $f \colon y \ra x$ in $X$ lying over a
  morphism $\Spec \phi \colon \Spec B \ra \Spec A$,
  the corresponding morphism $\mathcal{F}(f) \colon \mathcal{F}(x) \ra \mathcal{F}(y)$ 
  is equivariant with respect to $\Delta(A) \ra \Delta(B)$.
  Moreover, since $\mathcal{F}$ is quasi-coherent,
  it induces an isomorphism $B \otimes_{\phi, A} 
  \mathcal{F}(x) \sira \mathcal{F}(y)$ of $B$-modules.

  By restricting $\mathcal{F}$ to $\Spec A$ along $x\colon \Spec A \to X$ (see \cite[\sptag{06W9}]{stacks-project}),
  we see that the $A$-module $\mathcal{F}(x)$ can be viewed as
  a right representation of the diagonalizable $A$-group scheme
  $\Delta_{\Spec A}$ in the sense of \cite[I.2.7]{Jantzen}.
  Hence it decomposes as 
  $\mathcal{F}(x)=\bigoplus_{\chi \in A}\mathcal{F}(x)_\chi$ by 
  \cite[I.2.11]{Jantzen}.

  If the object $x$ varies, all these decompositions 
  are compatible and combine into the decomposition 
  $\mathcal{F}=\bigoplus_{\chi \in A} \mathcal{F}_\chi$ where
  $\mathcal{F}_\chi$ is the 
  $\chi$-homogeneous subsheaf of $\mathcal{F}$. 
  This shows that the functor \eqref{eq-eq-qc-gerbe} is
  essentially surjective.
  Similarly, we get a decomposition of morphisms between pairs of quasi-coherent modules
  (cf. \cite[I.2.11]{Jantzen}),
  which shows that \eqref{eq-eq-qc-gerbe} is fully faithful.
  Moreover, any $\Delta_S$-character $\chi$ clearly gives an exact functor $\mathcal{F} \mapsto \mathcal{F}_\chi$.
  In particular, $\Qcoh_\chi(X)$ is a Serre subcategory of $\Qcoh(X)$.

  The statement that the tensor product of a $\chi$-homogeneous
  quasi-coherent module with a $\psi$-homogeneous quasi-coherent
  module is $(\chi+\psi)$-homogeneous is obvious, and
  the corresponding claim for the internal Hom-functor is then a
  formal consequence.
\end{proof}

\begin{remark}
\label{rem-coprod-prod}
Note that the canonical morphism
$
\bigoplus_{\chi\in A} \mathcal{F}_\chi
\to
\prod_{\chi\in A} \mathcal{F}_\chi
$
is an isomorphism even if the character group
$A$ is not finite. Here the product is taken in $\Qcoh(X)$.
This is a simple consequence of 
the equivalence 
\eqref{eq-eq-qc-gerbe} in
Theorem~\ref{t-qc-gerbe-split}.
\end{remark}

\begin{proposition}
\label{p:pull-back-grading}
Let $S$ be an algebraic stack and 
$\alpha\colon X \to S$ a gerbe banded by a diagonalizable group
$\Delta$.
Let $f\colon T \to S$ be a morphism of algebraic stacks,
and let $g\colon Y \to X$ denote the base change of $f$ along
$\alpha$.
Then the functors
$$
g^*\colon \Qcoh(X) \to \Qcoh(Y), \qquad
g_*\colon \Qcoh(Y) \to \Qcoh(X)
$$
respect the decompositions of $\Qcoh(X)$ and~$\Qcoh(Y)$
into homogeneous objects given in Theorem~\ref{t-qc-gerbe-split}.
\end{proposition}
\begin{proof}
As a formal consequence of adjunction,
it is enough to verify that one of the functors $g^*$ and~$g_*$ preserves the decomposition.
The functor $g^*$ is obtained by restriction of the functor
$g^*\colon\Mod(S_\fppf, \mathcal{O}_S) \to \Mod(T_\fppf, \mathcal{O}_T)$,
which in turn is just given by restriction of the topos $S_\fppf$.
In this setting the statement is obvious from the definition of a
homogeneous sheaf
(see Definition~\ref{def-homogeneous-sheaf}).
\end{proof}

\begin{proposition}
  \label{t:equivalence-epi}
  Let $X$ and $X'$ be gerbes over $S$ banded by diagonalizable groups $\Delta$ and~$\Delta'$,
  respectively, and let $\rho\colon X \to X'$ be a morphism over~$S$ with induced morphism $\varphi\colon \Delta_S \to \Delta'_S$ on the bands.
  Fix a character $\chi'\colon \Delta'_S \to \GGms{S}$, and let $\chi =
  \chi' \circ \varphi$.
  Then the pull-back functor $\rho^*$ takes $\chi'$-homogeneous
  objects to $\chi$-homogeneous 
  objects and induces an equivalence
  \begin{equation}
    \label{eq-rho-restricted}
    \rho^*\colon 
    \Qcoh_{\chi'}(X') \sira \Qcoh_{\chi}(X)
  \end{equation}
  of categories.
  In particular, if $X' = S$ then $\rho^*\colon \Qcoh(S) \to \Qcoh(X)$ is fully faithful with 
  essential image consisting of those objects that are
  homogeneous for the trivial character. 
\end{proposition}

\begin{proof}
By a similar argument as in the proof of Proposition~\ref{p:pull-back-grading},
we see that $\rho^*$
takes $\chi'$-homogeneous objects to 
$\chi$-homogeneous objects.
First we note that the special case in the last
sentence in the statement of the proposition follows
directly from 
Proposition~\ref{p-gerbe-pb-ff}.

In the general case, we consider the fiber product $X\times_S X'$
and its projections $p$ and~$q$ to $X$ and~$X'$.
Note that $X\times_S X'$ is a $\Delta\times\Delta'$-gerbe over~$S$
and that $p$ and~$q$ induce the projections to $\Delta$ respectively~$\Delta'$
on the bands.

The projection $p$ is a gerbe banded by the kernel of the
projection $\Delta\times\Delta' \to \Delta$,
which is isomorphic to~$\Delta'$.
In particular, $p^*$ is fully faithful with essential image consisting of sheaves
on which this kernel acts trivially.
Similar statements hold for~$q$.
Hence it follows from the decomposition in Theorem~\ref{t-qc-gerbe-split}
that we get equivalences
\begin{align}
  \label{eq:equiv-p}
  p^* \colon \Qcoh_\chi(X)
  & \rightleftarrows \Qcoh_{(\chi,e')}(X \times_S X'/S)\colon p_*,\\
  \label{eq:equiv-q}
  q^* \colon \Qcoh_{\chi'}(X')
  & \rightleftarrows \Qcoh_{(e, \chi')}(X \times_S X'/S)\colon q_*,
\end{align}
where $e$ and~$e'$ denote the trivial characters of $\Delta$ and~$\Delta'$,
respectively.

Note that the gerbe $p$ has a trivialization~$\tau = (\id, \rho)$.
In particular, there exists a line bundle $\mathcal{L}$ on $X\times_S X'$ such
that
\begin{enumerate}
\item
\label{it-cond-1}
$\mathcal{L}$ is $\chi'$-homogeneous with respect to the decomposition induced by~$p$;
\item
\label{it-cond-2}
$\tau^*\mathcal{L} \cong \mathcal{O}_X$.
\end{enumerate}
Since $\mathcal{L}$ is a line bundle,
it must also be homogeneous for some character $(\psi, \psi')$ of $(\Delta\times\Delta')_S$
with respect to the gerbe $X\times_S X' \to S$.
By \ref{it-cond-1}, we have $\psi' = \chi'$.
Furthermore, since $\tau$ induces the morphism $(\id, \varphi)$ on bands,
we have $\tau^*\mathcal{L} \in \Qcoh_{\psi\chi}(X)$.
This forces $\psi = \chi^{-1}$ by \ref{it-cond-2},
so $\mathcal{L} \in \Qcoh_{(\chi^{-1}, \chi')}(X\times_S X'/S)$.
In particular, it follows from the equivalences \eqref{eq:equiv-p} and~\eqref{eq:equiv-q} and the statement about
the monoidal structure in Theorem~\ref{t-qc-gerbe-split} that we
get an equivalence
$$
p_*(\mathcal{L}^\vee\otimes q^*(-))\colon \Qcoh_{\chi'}(X') \sira \Qcoh_{\chi}(X).
$$
Furthermore, the restriction of the natural transformation 
\begin{multline*}
p_*(\mathcal{L}^\vee\otimes q^*(-)) \cong \tau^*p^*p_*(\mathcal{L}^\vee\otimes q^*(-)) \to\\
\tau^*(\mathcal{L}^\vee\otimes q^*(-)) \cong \tau^*(\mathcal{L}^\vee)\otimes \tau^*q^*(-) \cong \rho^*
\end{multline*}
to $\Qcoh_{\chi'}(X')$ is an isomorphism,
which concludes the proof.
\end{proof}

Now the following corollaries are easy consequences of Theorem~\ref{t-qc-gerbe-split},
Proposition~\ref{p:pull-back-grading} and~Proposition~\ref{t:equivalence-epi}.
\begin{corollary}
\label{cor:twisted-power}
Let $S$ be an algebraic stacks and let $\alpha\colon X \to S$ be a gerbe banded by $\Delta$,
and choose a character $\chi\colon \Delta_S \to \GGms{S}$.
Then we have an equivalence
\begin{equation}
\label{eq-pb-character}
\Qcoh_{\chi_*\alpha}(S)\cong \Qcoh_\chi(X) 
\end{equation}
of categories, where $\chi_*\alpha$ is defined as in Construction~\ref{cons-gerbe-pf}.
In the particular case when $\Delta = \GGm$,
this gives the equivalence
\begin{equation}
\label{eq-pb-power}
\Qcoh_{\alpha^d}(S)\cong \Qcoh_{d}(X) 
\end{equation}
for any $d \in \mathbb{Z}$, where $\alpha^d$ is defined as in Construction~\ref{cons-gerbe-operations}.
\end{corollary}
\begin{proof}
The equivalence \eqref{eq-pb-character} is the special case of Proposition~\ref{t:equivalence-epi} where
we take $\rho$ to be $X \to \chi_* X$ as in Construction~\ref{cons-gerbe-pf}.
The equivalence \eqref{eq-pb-power} is the special case of \eqref{eq-pb-character},
where we take $\chi = p_n$ as in Construction~\ref{cons-gerbe-operations}.
\end{proof}

\begin{corollary}
\label{cor:twisted-tensor}
Let $S$ be an algebraic stack and let $\alpha\colon X \to S$ and~$\beta\colon Y \to S$ be $\GGm$-gerbes.
Then the tensor product and the internal Hom-functor on $\Qcoh(X\times_S Y)$ induce functors
\begin{align*}
\otimes & \colon \Qcoh_{\alpha}(S) \times \Qcoh_{\beta}(S) \to \Qcoh_{\alpha\beta}(S), \\
\sheafHom & \colon \Qcoh_{\alpha}(S)^\opp \times \Qcoh_{\beta}(S)
\to \Qcoh_{\alpha^{-1}\beta}(S).
\end{align*}
Furthermore, if $f\colon T \to S$ is an arbitrary morphism of algebraic stacks,
then we get an induced pair of adjoint functors
\begin{equation*}
  f^* \colon \Qcoh_{\alpha}(S) \to \Qcoh_{f^*\alpha}(T),
  \qquad
  f_* \colon \Qcoh_{f^*\alpha}(T) \to \Qcoh_{\alpha}(S).
\end{equation*}
Recall our conventions regarding $f_*$ and $\sheafHom$ from Notation~\ref{not-sheaf-op}.
\end{corollary}
\begin{proof}
The product $X\times_S Y$ has the structure of a $\GGm\times \GGm$-gerbe
such that the projections on the factors $X$ and $Y$ induce the projections on the factors
on the bands.
By applying Proposition~\ref{t:equivalence-epi} to the two projections $\GGm\times \GGm \to \GGm$,
we get equivalences
$$
\Qcoh_{\alpha}(S) \cong \Qcoh_{(1, 0)}(X\times_S Y), \qquad
\Qcoh_{\beta}(S) \cong \Qcoh_{(0, 1)}(X\times_S Y),
$$
where we have identified the character group of $\GGm$ with~$\mathbb{Z}$.
By applying Proposition~\ref{t:equivalence-epi} to the multiplication and inversion maps
given in Construction~\ref{cons-gerbe-operations},
we obtain the equivalences
$$
\Qcoh_{\alpha\beta}(S) \cong \Qcoh_{(1, 1)}(X\times_S Y), \qquad
\Qcoh_{\alpha^{-1}\beta}(S) \cong \Qcoh_{(-1, 1)}(X\times_S Y),
$$
respectively.
Now the functors $\otimes$ and $\sheafHom$ are obtained from the monoidal structure of $\Qcoh(X\times_S Y)$
described in Theorem~\ref{t-qc-gerbe-split}.
The statement about the functors $f^*$ and $f_*$ is an obvious consequence of Proposition~\ref{p:pull-back-grading}.
\end{proof}


\section{The derived category of a gerbe}
\label{sec:decomp-deriv-categ}
In this section we generalize some of the results by C{\u{a}}ld{\u{a}}raru
\cite[Chapter~2]{caldararu2000} and Lieblich \cite[Section~2.2.4]{lieblich2004}
on the basic structure of derived categories of gerbes and twisted sheaves.
Throughout the section,
we let $\alpha\colon X \to S$ be a gerbe banded by a diagonalizable group $\Delta$
with character group $A$.

\begin{definition}
\label{def:derived-twisted}
Given a character $\chi \colon \Delta_S \to \GGms{S}$,
we say that an object $\mathcal{F}$ in $\D_\qc(X)$ is
\define{$\chi$-homogeneous} if it has $\chi$-homogeneous
cohomology.
We denote the full subcategory of $\chi$-homogeneous objects in $\D_\qc(X)$ by~$\D_{\qc, \chi}(X)$.

Consider the special case $\Delta = \GGm$ and let $\id\colon \GGm \to \GGm$ be the identity character.
An object in $\D_{\qc, \id}(X)$ is called an \emph{$\alpha$-twisted complex on $S$}.
We often use the notation $\D_{\qc, \alpha}(S)$ instead of~$\D_{\qc, \id}(X)$
employing a similar abuse of notation as in Remark~\ref{rem-twisted-abuse}
if $\alpha$ is just a cohomology class in $\HH^2(S_\fppf, \GGm)$.
\end{definition}

We state the following lemma for convenient reference.
See Remark~\ref{rem-trivial-case} for sharper results.

\begin{lemma}
\label{lem:affine-derived-decomp}
Let $S$ be an affine scheme and let $\alpha\colon X \to S$ be a trivial $\Delta$-gerbe,
where  $\Delta$ is a diagonalizable group.
Then the obvious functor
\begin{equation}
\label{eq-equiv-qcoh}
\D(\Qcoh(X)) \xrightarrow{\sim} \D_\qc(X)
\end{equation}
is a monoidal equivalence and induces equivalences
\begin{equation}
\label{eq-equiv-qcoh-chi}
\D(\Qcoh_\chi(X)) \xrightarrow{\sim} \D_{\qc, \chi}(X)
\end{equation}
for each character $\chi$ of $\Delta$.
Moreover, let $T$ be an affine scheme and let $\beta\colon X' \to T$ be a trivial $\Delta'$-gerbe,
where $\Delta'$ is a diagonalizable group.
Assume that we have a morphism $\rho\colon X \to X'$ of
algebraic stacks.
Then the squares
\begin{equation}
\label{eq-commute-derived}
\xymatrix{
\D(\Qcoh(X')) \ar[r]^-{\sim}\ar[d]_{\Ld\rho^*} & \D_\qc(X') \ar[d]^{\rho^*}\\
\D(\Qcoh(X)) \ar[r]^-{\sim}& \D_\qc(X), \\
}\qquad
\xymatrix{
\D(\Qcoh(X')) \ar[r]^-{\sim} & \D_\qc(X') \\
\D(\Qcoh(X)) \ar[r]^-{\sim}\ar[u]^{\Rd\rho_*}& \D_\qc(X)\ar[u]_{\rho_*} \\
}
\end{equation}
commute up to natural equivalence.
\end{lemma}

\begin{proof}
The equivalence \eqref{eq-equiv-qcoh} is a special case of \cite[Proposition~2.2.4.6]{lieblich2004},
and the equivalences \eqref{eq-equiv-qcoh-chi} are then immediate
from Theorem~\ref{t-qc-gerbe-split}.
The fact that the equivalence \eqref{eq-equiv-qcoh} is monoidal follows from the fact that
any complex in $\Qcoh(X)$ admits a resolution 
by an h-flat complex in $\Qcoh(X)$ which is also h-flat as a
complex in $\Mod(X_\liset, \mathcal{O}_X)$.
This can be seen by using a homotopy limit argument (see~\cite[\sptag{06XX}]{stacks-project}),
using that any object in $\Qcoh(X)$ is a quotient of a locally free sheaf and the fact that coproducts in $\Mod(X_\liset, \mathcal{O}_X)$ are exact and preserve $\Qcoh(X)$.

The functor $\alpha_*\colon \Qcoh(X) \to \Qcoh(S)$ is exact,
since this operation corresponds to taking invariants and $\Delta$ is diagonalizable.
This well-known fact also easily follows from the last sentence of Proposition~\ref{t:equivalence-epi}
together with the decomposition in~Theorem~\ref{t-qc-gerbe-split}.
Hence the equivalence \eqref{eq-equiv-qcoh} shows that $X$ has cohomological dimension zero.

Note that $\rho$ induces a morphism $S \to T$ on the coarse spaces,
so $\rho$ factors as $X \to X'\times_T S \to X'$ where the first morphism
is a morphism of gerbes over $S$, and the second morphism is affine.
Furthermore, the first morphism factors as an affine morphism followed by a
$\Delta$-gerbe by Remark~\ref{rem-gerbe-ff}.
Hence also $\rho$ has cohomological dimension zero,
as this property is stable under composition and can be checked after a faithfully flat base change
by \cite[Lemma~2.2]{hr2017}).
In particular, $\rho$ is concentrated, so it follows from \cite[Corollary~2.2]{hnr2018} that the right hand
square \eqref{eq-commute-derived} commutes.
Hence also the left hand square commutes by adjunction.
\end{proof}

\begin{remark}
\label{rem-gerbe-concentrated}
Let $\rho\colon X \to Y$ be a morphism of gerbes banded by diagonalizable groups
over an algebraic stack $S$ as in Remark~\ref{rem-gerbe-ff}.
Then the last part of the proof shows that $\rho$ is concentrated of cohomological dimension zero,
as these properties can be checked locally on~$S$
(see \cite[Lemma~2.2 and Lemma~2.5]{hr2017}).
\end{remark}

The next theorem is the main structure theorem for the derived
category of a $\Delta$-gerbe and the derived analog of Theorem~\ref{t-qc-gerbe-split}.
It is a generalization of the observation in \cite[Section~2.2.4]{lieblich2004}
that the derived category splits according to characters in certain situations.
\begin{theorem}
  \label{t:chi-splitting-D-qc-X}
  Let $S$ be an algebraic stack and 
  $\alpha\colon X \to S$ a gerbe banded by a diagonalizable group
  $\Delta$.
  Then $\D_{\qc, \chi}(X)$ is
  a triangulated subcategory of $\D_{\qc}(X)$ for each $\Delta_S$-character $\chi$,
  and taking the coproduct gives an equivalence
\begin{equation}
\label{eq-equivalence-derived}
\prod_{\chi \in A} \D_{\qc, \chi}(X) \sira \D_\qc(X),
\qquad
 (\mathcal{F}_\chi)_{\chi \in A} 
\mapsto 
\bigoplus_{\chi \in A} \mathcal{F}_\chi,
\end{equation}
of triangulated categories, where $A$ denotes the character group of $\Delta$.

Furthermore, assume that $\mathcal{F}_\chi$ and $\mathcal{G}_\psi$ are objects in $\D_\qc(X)$
which are homogeneous for $\Delta_S$-characters $\chi$ and $\psi$, respectively.
Then the objects
\begin{equation}
\label{eq-der-monoidal}
\mathcal{F}_\chi \otimes \mathcal{G}_\psi, \qquad
\sheafHom(\mathcal{G}_\psi, \mathcal{F}_\chi)
\end{equation}
are $\chi\psi$- and $\chi\psi^{-1}$-homogeneous, respectively.
\end{theorem}

\begin{proof}
By Theorem \ref{t-qc-gerbe-split},
the category $\Qcoh_\chi(X)$ is a Serre subcategory of $\Qcoh(X)$.
It follows that $\D_{\qc, \chi}(X)$ is a strictly full triangulated subcategory of $\D_\qc(X)$
by \cite[\sptag{07B4}, \sptag{06UQ}]{stacks-project}.

Recall that we work with sheaves on the lisse--étale sites over $X$ and~$S$.
Note that we get a morphism $X_{\liset} \to S_{\liset}$ of topoi,
by 
\cite[\sptag{07AT}]{stacks-project},
since $\alpha$ is smooth (see Remark~\ref{rem-gerbe-ff}).
In particular, there is no need to work with hypercoverings to define pull-backs as in~\cite{olsson2007}.

Let $E$ be a subset of the set $A$ of characters of $\Delta$,
and let $\Qcoh_E(X)$ denote the full subcategory of $\Qcoh(X)$ of objects
which are direct sums of $\chi$-homogeneous objects for $\chi \in E$.
Let $F$ be the complement of $E$ in $A$.
It follows from Theorem~\ref{t-qc-gerbe-split}
that $(\Qcoh_E(X), \Qcoh_F(X))$ is a torsion pair in $\Qcoh(X)$ and that both
$\Qcoh_E(X)$ and~$\Qcoh_F(X)$ are abelian subcategories of~$\Qcoh(X)$.

Let $\mathcal{E}$ and $\mathcal{F}$ be objects of $\Qcoh_E(X)$ and~$\Qcoh_F(X)$,
respectively.
We will show that
\begin{equation}
\label{eq-ext-vanishing}
\Ext^n_{\mathcal{O}_X}(\mathcal{E}, \mathcal{F})
\cong \HH^n(\Rd\Gamma\Rd\sheafHom_{\mathcal{O}_X}(\mathcal{E}, \mathcal{F}))
\end{equation}
vanishes for each $n$.
Since
$
\Rd\Gamma\Rd\sheafHom_{\mathcal{O}_X}(\mathcal{E}, \mathcal{F})
\cong
\Rd\Gamma\Rd\alpha_\ast\Rd\sheafHom_{\mathcal{O}_X}(\mathcal{E}, \mathcal{F}),
$
this follows if we prove that $\Rd\alpha_\ast\Rd\sheafHom_{\mathcal{O}_X}(\mathcal{E}, \mathcal{F})$ vanishes.
The $n$-th cohomology sheaf of $\Rd\alpha_\ast\Rd\sheafHom_{\mathcal{O}_X}(\mathcal{E}, \mathcal{F})$ is easily seen
to be the sheaf of $\mathcal{O}_S$-modules associated to the presheaf
\begin{equation*}
(U \to S) \mapsto
\Ext^n_{\mathcal{O}_{X_U}}(\mathcal{E}|_{X_U}, 
\mathcal{F}|_{X_U})
\end{equation*}
where $U$ is a scheme and $U \ra S$ is smooth. 
It is therefore enough to verify the vanishing of \eqref{eq-ext-vanishing} when $S$ is affine and the gerbe
$\alpha\colon X \to S$ is trivial.
But in this situation the functor \eqref{eq-equiv-qcoh} is an equivalence by Lemma~\ref{lem:affine-derived-decomp},
so we get the desired vanishing from
Lemma~\ref{p:abelian-torsion-pair}.\ref{enum:Ext-TF} or, more
easily, from
Theorem~\ref{t-qc-gerbe-split}.

All direct and inverse truncation systems in $\D_\qc(X)$ are
effective by Propositions~\ref{p:ab4-trunc} and \ref{p:Dqc-truncation-complete}.
Hence the hypotheses of Theorem~\ref{t:sod-DBA} are satisfied,
so we get a semi-orthogonal decomposition
\begin{equation}
\label{eq-so-decomp}
\D_\qc(X) = \langle \D_{\qc, F}(X), \D_{\qc, E}(X)\rangle,
\end{equation}
where $\D_{\qc, E}$ denotes the full subcategory of $\D_\qc(X)$ of objects with cohomology in
$\Qcoh_E(X)$ and similarly for $\D_{\qc, F}(X)$.

We are now ready to prove that the functor \eqref{eq-equivalence-derived} is an equivalence.
First note that it is well defined since $\D_\qc(X)$ has arbitrary coproducts.
Let $\chi \in A$ and set
$E = \{\chi\}$, $F = A -\{\chi\}$.
Then $\Hom_{\D_\qc(X)}(\mathcal{F}_\chi, \mathcal{G})$ vanishes
whenever $\mathcal{F}_\chi \in \D_{\qc, \chi}(X)$
and $\mathcal{G} \in \D_{\qc, F}$ by the 
semi-orthogonal decomposition~\eqref{eq-so-decomp}.
Since $\D_{\qc, \chi}(X)$ is a full subcategory of $\D_\qc(X)$,
it follows that the functor \eqref{eq-equivalence-derived} is fully faithful.

Now let $\mathcal{F}$ be any object of $\D_\qc(X)$.
Then the semi-orthogonal decomposition~\eqref{eq-so-decomp}
gives a triangle
$$
\mathcal{F}_\chi \to \mathcal{F} \to \mathcal{F}_F \to \Sigma\mathcal{F}_\chi
$$
with $\mathcal{F}_\chi \in \D_{\qc, \chi}(X)$ and $\mathcal{F}_F
\in \D_{\qc, F}(X)$, respectively.
By Remark~\ref{rem:cohomology-split} 
and Theorem~\ref{t-qc-gerbe-split},
the morphism $\mathcal{F}_\chi \to \mathcal{F}$ induces an isomorphism
$\HH^n(\mathcal{F}_\chi) \to \HH^n(\mathcal{F})_\chi$ for each $n$.
It follows that the canonical map $\bigoplus_{\chi \in A} \mathcal{F}_\chi \to \mathcal{F}$ induces
an isomorphism on cohomology and therefore is an isomorphism in $\D_\qc(X)$.
Hence the functor \eqref{eq-equivalence-derived} is essentially
surjective and therefore an equivalence. 

Finally, we prove the statements about the monoidal structure.
It is enough to prove the statement about the tensor product,
since the statement about the internal hom follows formally from the
adjointness property.
The question whether $\mathcal{F}_\chi \otimes_\mathcal{O_X}^\Ld \mathcal{G}_\psi$
is $\chi\psi$-homogeneous can be verified locally on $S$.
Indeed, this follows from Proposition~\ref{p:pull-back-grading},
since cohomology of a complex commutes with pull-back along a flat map (see \cite[Equation~1.9]{hr2017}),
and the pull-back functor along a faithfully flat map is conservative.
Hence we may assume that $S$ is affine and that the gerbe $\alpha\colon X \to S$
is trivial.
Now the statement follows from the underived case described in Theorem~\ref{t-qc-gerbe-split}
by the monoidal decomposition preserving
equivalence \eqref{eq-equiv-qcoh} in Lemma~\ref{lem:affine-derived-decomp}.
\end{proof}

\begin{remark}
Similarly as in Remark~\ref{rem-coprod-prod}, the canonical morphism
$
\bigoplus_{\chi\in A} \mathcal{F}_\chi
\to
\prod_{\chi\in A} \mathcal{F}_\chi
$
is an isomorphism even if the character group
$A$ is not finite.
\end{remark}

The following two propositions are the derived analogs of
Proposition~\ref{p:pull-back-grading} and \ref{t:equivalence-epi}.

\begin{proposition}
\label{p:pull-back-grading-derived}
Let $S$ be an algebraic stack and 
$\alpha\colon X \to S$ a gerbe banded by a diagonalizable group
$\Delta$.
Let $f\colon T \to S$ be a morphism of algebraic stacks,
and let $g\colon Y \to X$ denote the canonical morphism from the pull-back
$Y$ of $X$ along $f$.
Then the functors
$$
g^*\colon \D_\qc(X) \to \D_\qc(Y), \qquad
g_*\colon \D_\qc(Y) \to \D_\qc(X)
$$
respect the decompositions of $\D_\qc(X)$ and~$\D_\qc(Y)$
into homogeneous components given in Theorem~\ref{t:chi-splitting-D-qc-X}.
\end{proposition}
\begin{proof}
As a formal consequence of adjunction,
it is enough to verify that the pull-back $g^*$ respects the decomposition.
By a similar reduction argument as in the last paragraph of
the proof of Theorem~\ref{t:chi-splitting-D-qc-X},
we reduce to the case when $S$ and $T$ are affine and the gerbes are trivial.
Now the result follows from the underived setting in Proposition~\ref{p:pull-back-grading}
and the commutative diagram \eqref{eq-commute-derived} in Lemma~\ref{lem:affine-derived-decomp}.
\end{proof}

\begin{proposition}
  \label{t:equivalence-epi-derived}
  Let $X$ and $X'$ be gerbes over $S$ banded by diagonalizable groups $\Delta$ and~$\Delta'$,
  respectively, and let $\rho\colon X \to X'$ be a morphism over~$S$ with induced morphism $\varphi\colon \Delta_S \to \Delta'_S$ on the bands.
  Fix a character $\chi'\colon \Delta'_S \to \GGms{S}$, and let $\chi =
  \chi' \circ \varphi$.
  Then the pull-back functor $\rho^*$ takes $\chi'$-homogeneous
  objects to $\chi$-homogeneous 
  objects and induces an equivalence
  \begin{equation}
    \label{eq:restricted-pb-derived}
    \rho^*\colon \D_{\qc,\chi'}(X') \sira \D_{\qc,\chi}(X)
  \end{equation}
  of triangulated categories.
  In particular, if $X' = S$ then $\rho^*\colon \D_\qc(S) \to \D_\qc(X)$ is fully faithful with 
  essential image consisting of those objects that are
  homogeneous for the trivial character. 
\end{proposition}

\begin{proof}
Since $\rho$ is flat (see Remark~\ref{rem-gerbe-ff}),
the pull-back $\rho^*$ commutes with
taking cohomology (see~\cite[Equation~(1.9)]{hr2017}).
In particular,
the first statement follows from the underived case,
which is given by Proposition~\ref{t:equivalence-epi}.

Let us first consider the special case $X' = S$.
By Remark~\ref{rem-gerbe-concentrated}, the morphism $\rho$ is concentrated,
so the push-forward $\rho_*$ respects flat base change by \cite[Theorem~2.6(4)]{hr2017}.
Hence the statement can be verified locally.
In particular, we may assume that the gerbes are trivial and that $S$ is affine.
Hence we may use the equivalences in Lemma~\ref{lem:affine-derived-decomp}.
Since $\rho^*$ and $\rho_*$ are exact,
the statement follows from the underived case,
which is Proposition~\ref{t:equivalence-epi}.

For the general case, we proceed exactly as in
the proof of Proposition~\ref{t:equivalence-epi} on the level of derived
categories.
\end{proof}

Similarly as in the underived case, we get the following corollaries.
\begin{corollary}
\label{cor:twisted-power-derived}
Let $S$ be an algebraic stacks and let $\alpha\colon X \to S$ be a gerbe banded by $\Delta$,
and choose a character $\chi\colon \Delta_S \to \GGms{S}$.
Then we have an equivalence
\begin{equation*}
\D_{\qc,\chi_*\alpha}(S)\cong \D_{\qc, \chi}(X) 
\end{equation*}
of categories, where $\chi_*\alpha$ is defined as in Construction~\ref{cons-gerbe-pf}.
In the particular case when $\Delta = \GGm$,
this gives the equivalence
\begin{equation*}
\D_{\qc,\alpha^d}(S)\cong \D_{\qc,d}(X) 
\end{equation*}
for any $d \in \mathbb{Z}$, where $\alpha^d$ is defined as in Construction~\ref{cons-gerbe-operations}.
\end{corollary}
\begin{proof}
This is proven similarly as Corollary~\ref{cor:twisted-power} by using Proposition~\ref{t:equivalence-epi-derived}
instead of Proposition~\ref{t:equivalence-epi}.
\end{proof}

\begin{corollary}[{cf.~\cite[Theorem 2.2.4]{caldararu2000}}]
\label{cor:twisted-tensor-derived}
Let $S$ be an algebraic stack and let $\alpha\colon X \to S$ and~$\beta\colon Y \to S$ be $\GGm$-gerbes.
Then the tensor product and the internal Hom-functor on $\D_\qc(X\times_S Y)$ induce triangulated functors
\begin{align*}
\otimes & \colon \D_{\qc, \alpha}(S) \times \D_{\qc, \beta}(S)
\to \D_{\qc, \alpha\beta}(S),\\
\sheafHom & \colon \D_{\qc, \alpha}(S)^\opp \times \D_{\qc, \beta}(S)
\to \D_{\qc, \alpha^{-1}\beta}(S).
\end{align*}
Furthermore, if $f\colon T \to S$ is an arbitrary morphism of algebraic stacks,
then we get an induced pair of adjoint functors
\begin{equation*}
  f^* \colon \D_{\qc, \alpha}(S) \to \D_{\qc, f^*\alpha}(T),
  \qquad
  f_* \colon \D_{\qc, f^*\alpha}(T) \to \D_{\qc, \alpha}(S).
\end{equation*}
\end{corollary}
\begin{proof}
This is proven similarly as Corollary~\ref{cor:twisted-tensor} by using 
Theorem~\ref{t:chi-splitting-D-qc-X},
Proposition~\ref{p:pull-back-grading-derived}
and~Proposition~\ref{t:equivalence-epi-derived}
instead of
Theorem~\ref{t-qc-gerbe-split},
Proposition~\ref{p:pull-back-grading}
and~Proposition~\ref{t:equivalence-epi}.
\end{proof}

We conclude the section with two remarks regarding alternative
approaches for proving Theorem~\ref{t:chi-splitting-D-qc-X}.
\begin{remark}
\label{rem-trivial-case}
Note that Theorem~\ref{t:chi-splitting-D-qc-X} is trivial whenever the functor \eqref{eq-equiv-qcoh}
is a monoidal equivalence.
By Lieblich \cite[Proposition~2.2.4.6]{lieblich2004},
the functor \eqref{eq-equiv-qcoh} is an equivalence whenever $S$ is a quasi-compact,
separated scheme.
More generally, it follows from the work by Hall, Neeman and~Rydh \cite[Theorem~1.2]{hnr2018}
that \eqref{eq-equiv-qcoh} is an equivalence whenever
$S$ is an algebraic space which is either noetherian or quasi-compact with affine diagonal.
The aforementioned theorem applies since in this case,
the category $\D_\qc(X)$ is compactly generated by $|A|$ objects by
\cite[Theorem~6.9]{hr2017} 
applied to the presheaf
$
(T \to S) \mapsto \D_\qc(X_T)
$
of triangulated categories.
Since we are not going to use this result,
we leave the details of verifying the hypotheses of the
cited theorem to
the reader (cf.~\cite[Example~9.2]{hr2017}).
We do not know whether the equivalence \eqref{eq-equiv-qcoh} is monoidal in this generality.

Finally, it should be emphasized that there certainly are interesting cases when
\eqref{eq-equiv-qcoh} is not an equivalence. 
For instance, consider the $\GGm$-gerbe in Example~\ref{ex-gerbes}.\ref{it-universal-projective}.
The functor \eqref{eq-equiv-qcoh} is not an
equivalence for $X = \BB\GL_n$ if $n \geq 2$
by \cite[Theorem~1.3]{hnr2018},
so Theorem~\ref{t:chi-splitting-D-qc-X} seems to be non-trivial even in this basic case.
\end{remark}

\begin{remark}
\label{rem:decomposition-O-modules}
Theorem~\ref{t:chi-splitting-D-qc-X} would be trivial if
the category $\Mod(X_\liset, \mathcal{O}_X)$ admitted a decomposition
similar 
to the one for $\Qcoh(X)$ described in Theorem~\ref{t-qc-gerbe-split}.
We do not know whether this is the case,
but we suspect it is not.

In \cite[Exposé 1, Proposition~4.7.2]{sga-3-1-2011}, it is stated
that the category of (not necessarily quasi-coherent) sheaves of $G-\mathcal{O}_S$-modules
on the small Zariski site
is equivalent to the category of sheaves of $\mathcal{A}$-comodules for any affine group scheme $G = \Spec_S \mathcal{A}$ over
a scheme $S$.
If this were true, then it should be possible to adapt the proof
to our situation and to get a decomposition of $\Mod(X_\liset, \mathcal{O}_X)$.
However, we suspect that the previously cited proposition might be wrong, 
since it uses \cite[Exposé 1, Proposition~4.6.4]{sga-3-1-2011}
which in turn uses the fact that the canonical morphism
\begin{equation}
  \label{eq:qcoh-push-pull}
  \mathcal{F}\otimes_S \mathcal{A} \to \alpha_\ast\alpha^\ast\mathcal{F}
\end{equation}
is an isomorphism whenever $\mathcal{F}$ is quasi-coherent and
$\alpha\colon X = \Spec_S \mathcal{A} \to S$ is affine.
But this is certainly not true in general if we drop the condition that $\mathcal{F}$ be quasi-coherent. 

Let for example $S$ be a scheme with an open subscheme $U$
such that the inclusion $j \colon U \ra S$ 
is affine and quasi-compact
and $\mathcal{F}=j_! \mathcal{O}_U$ is not
quasi-coherent. Then the canonical morphism $j_!
\mathcal{O}_U \cong \mathcal{F} \otimes j_*\mathcal{O}_U \ra
j_*j^* \mathcal{F} \cong j_*\mathcal{O}_U$ is not an
isomorphism because $j_*\mathcal{O}_U$ is quasi-coherent.
Let $G:=S \sqcup U$ be the disjoint union. Then
$\alpha \colon G \ra S$ is in the obvious way an
affine \'etale group
scheme over $S$. In this case 
$\mathcal{A}=\alpha_*\mathcal{O}_G=
\mathcal{O}_S \oplus
j_*\mathcal{O}_U$ and the morphism
\eqref{eq:qcoh-push-pull} is not an isomorphism for 
$\mathcal{F}=j_! \mathcal{O}_U$.
\end{remark}


\section{Semi-orthogonal decompositions for Brauer--Severi schemes}
\label{sec:brau-severi-vari}
In this section, we demonstrate how to use the theory of gerbes developed by Giraud to obtain a
simple and conceptually appealing proof for the existence of the semi-orthogonal decomposition
of the derived category of a Brauer--Severi scheme given by Bernardara~\cite[Theorem~4.1]{bernardara-BS-schemes}.
We start by recalling some basic facts on Brauer--Severi schemes and their trivializing gerbes
from \cite[Example~V.4.8]{giraud1971}.

Let $S$ be an arbitrary algebraic stack.
An algebraic stack $P$ over $S$ is called a \emph{Brauer--Severi scheme} over $S$
if it is \emph{fppf}-locally (on $S$) isomorphic to a projectivized
vector bundle.
Of course $P$ need not be a scheme in general;
the terminology is motivated by the fact that the structure
morphism to $S$ is
representable by schemes. 

Given a Brauer--Severi scheme $\pi\colon P \to S$, we consider
its \emph{gerbe of trivializations} and denote it by
$\beta\colon B \to S$.  Let $t\colon T \to S$ be a morphism.  The
$t$-points of $B$ are pairs $(\mathcal{V}, v)$, where
$\mathcal{V}$ is a finite locally free sheaf of
$\mathcal{O}_T$-modules 
and
$v\colon P_{T} \xrightarrow{\sim} \mathbb{P}(\mathcal{V}^\vee)$
is an isomorphism over $T$.
The gerbe $B$ is endowed with a tautological
finite locally free sheaf $\mathcal{E}$ of $\mathcal{O}_{B}$-modules.
This is obtained on $t$-points $(\mathcal{V}, v)$ as above by
simply forgetting the isomorphism $v$.  The gerbe $B$ also has an
obvious banding by $\GGm$ given on $t$-points as above by the
usual action by $\mathcal{O}^\times_{T}(T)$ on $\mathcal{V}$. 
In particular,
the tautological
sheaf $\mathcal{E}$ is a $\beta$-twisted sheaf over $S$ in the
sense of Definition~\ref{def:twisted}.

Note that by construction, the gerbe of trivializations fits into a 2-cartesian diagram
\begin{equation}
  \label{eq:tautological-triv}
  \xymatrix{
    {\bP(\mathcal{E}^\vee)} \ar[r]^-{\gamma}\ar[d]_-{\rho} 
    & 
    {P} \ar[d]^-\pi\\
    {B} \ar[r]_-\beta 
    & {S,}
  }
\end{equation}
which we call the \emph{tautological trivialization diagram}.

The class $[\beta]$ of $\beta \colon B \ra S$ (see Theorem~\ref{t:gerbes-H2}) in $\HH^2(S_\fppf, \GGm)$
is called the \emph{Brauer class} of the Brauer--Severi scheme $\pi\colon P \to S$.
This class vanishes if and only if~$P$
is a projectivized vector bundle over~$S$.

\begin{remark}
\label{rem-universal-gerbe-of-trivializations}
Consider the Brauer--Severi scheme $\pi\colon [\bP^{n - 1}/\PGL_n] \to \BB \PGL_n$.
Its gerbe of trivializations can be identified with the $\GGm$-gerbe
$\alpha \colon \BB \GL_n \to \BB \PGL_n$
in Example~\ref{ex-gerbes}.\ref{it-universal-projective}.
Note that the stack $\BB \PGL_n$ classifies Brauer--Severi schemes.
Given a Brauer--Severi scheme $P$ over an arbitrary algebraic stack $S$,
its gerbe of trivializations can be identified with the pull-back of $\alpha$
along the morphism $S \to \BB \PGL_n$ corresponding to $P$.
\end{remark}

\begin{theorem}
  \label{t:sod-brauer-severi}
  Let $S$ be an algebraic stack and $\pi\colon P \to S$ a
  Brauer--Severi scheme of relative dimension $n \geq 0$ over $S$.
  Using the notation from above,
  we consider the tautological trivialization diagram~\eqref{eq:tautological-triv}.
  Then the functors 
  \begin{equation}
  \label{eq:BS-full-faithful}
  \Phi_i \colon \D_{\qc, i}(B) \to \D_{\qc}(P), \qquad
  \mathcal{F} \mapsto \gamma_*(\mathcal{O}(i)\otimes \rho^*\mathcal{F}),
  \end{equation}
  are fully faithful for each $i$.
  Moreover, for each $a \in \mathbb{Z}$,
  we have a semi-orthogonal decomposition
  \begin{equation}
  \label{eq:BS-sod}
  \D_\qc(P) = \langle \im \Phi_a, \ldots, \im \Phi_{a + n} \rangle
  \end{equation}
  into right admissible categories.
\end{theorem}
\begin{remark}
  \label{r:image-twisted}
  In particular (see~Corollary~\ref{cor:twisted-power-derived}),
  the essential image $\im \Phi_i$ is equivalent to the
  category $\D_{\qc, \beta^i}(S)$ of
  $\beta^i$-twisted complexes on $S$ 
  in the sense of Definition~\ref{def:derived-twisted}.
  Recall that, as explained above,
  $[\beta] \in \HH^2(S_\fppf, \GGm)$ is the Brauer class of $P$.
\end{remark}
\begin{proof}
Since $\rho$ is the projectivization of a vector bundle of rank $n + 1$,
it follows from \cite[Theorem~6.7]{BS-conservative} that the functors
$$
\Psi_i \colon \D_\qc(B) \to \D_{\qc}(\mathbb{P}(\mathcal{E}^\vee)), \qquad
\mathcal{F} \mapsto \mathcal{O}(i)\otimes \rho^*\mathcal{F},
$$
are fully faithful and that
$$
\D_\qc(\mathbb{P}(\mathcal{E}^\vee)) = \langle \im \Psi_{a}, \ldots,  \im \Psi_{a + n}\rangle
$$
is a semi-orthogonal decomposition for any $a \in \mathbb{Z}$.
Since both $\beta$ and $\gamma$ are $\GGm$-gerbes,
Theorem~\ref{t:chi-splitting-D-qc-X} provides equivalences
  \begin{equation*}
    \prod_{d \in \bZ} \D_{\qc, d}(B) 
    \sira 
    \D_\qc(B),
    \qquad
    \prod_{d \in \bZ} \D_{\qc, d}(\bP(\mathcal{E}^\vee)) 
    \sira
    \D_\qc(\bP(\mathcal{E}^\vee)).
  \end{equation*}
By Proposition~\ref{p:pull-back-grading-derived},
the pull-back $\rho^*$ preserves $d$-homogeneous objects for each
$d \in \bZ$.
Since $\mathcal{E}^\vee$ is homogeneous of degree -1 we see from
the canonical surjection $\rho^*(\mathcal{E}^\vee) \to
\mathcal{O}(1)$ and~Theorem~\ref{t:chi-splitting-D-qc-X} that
$\mathcal{O}(1)$ is homogeneous of degree -1. This shows that
$\Psi_i$ shifts the degree by $-i$.
Hence we get a semi-orthogonal decomposition
$$
\D_{\qc, 0}(\mathbb{P}(\mathcal{E}^\vee)) = \langle \im \Psi'_{a}, \ldots,  \im \Psi'_{a + n}\rangle,
$$
where $\Psi'_i$ denotes the restriction of $\Psi_i$ to $\D_{\qc, i}(B)$.
Since $\gamma$ is a gerbe, the pull-back $\gamma^*$ is fully faithful with essential image
$\D_{\qc, 0}(\mathbb{P}(\mathcal{E}^\vee))$ by
Proposition~\ref{t:equivalence-epi-derived}.
Therefore the theorem follows from the identities $\Phi_i = \gamma_*\circ\Psi'_i$.
\end{proof}

\begin{remark}
Note that Theorem~\ref{t:sod-brauer-severi} differs from the corresponding decomposition
in \cite[Theorem~4.1]{bernardara-BS-schemes} with respect to the sign of the Brauer class $\beta$.
This is simply due to different sign conventions.
We follow the convention used by Giraud in \cite[Example~V.4.8]{giraud1971}.
\end{remark}

Similarly, as in the corresponding theorem for projectivized vector bundles
(see \cite[Theorem~6.7, Corollary~6.8]{BS-conservative}),
we also get semi-orthogonal decompositions for the categories of
perfect complexes, of
locally bounded pseudo-coherent complexes, and for the singularity categories, respectively.
Following the notation from \cite[Section~2.3]{BS-conservative},
we denote these categories by $\D_\pf(P)$, $\D^\locbd_\pc(P)$ and~$\D^\sg_\pc(P)$,
respectively.

\begin{corollary}
\label{c:sod-brauer-severi}
Keep the assumptions from Theorem~\ref{t:sod-brauer-severi}.
The functors $\Phi_i$ induce functors
\begin{align*}
\Phi^\pf_i & \colon \D_{\pf, i}(B) \to \D_{\pf}(P), \qquad
\Phi^\pc_i \colon \D^\locbd_{\pc, i}(B) \to \D^\locbd_{\pc}(P),\\
\Phi^\sg_i & \colon \D_{\sg, i}(B) \to \D_{\sg}(P),
\end{align*}
where the source categories are the full subcategories of
$i$-homogeneous objects in the categories
$\D_{\pf}(B)$, $\D^\locbd_{\pc}(B)$ and~$\D_\sg(B)$ introduced above,
respectively.
Furthermore, for each $a \in \mathbb{Z}$, this yields a semi-orthogonal decomposition
\begin{align}
\label{eq-sod-perfect}
\D_\pf(P) & = \langle \im \Phi^\pf_a, \ldots, \Phi^\pf_{a + n}\rangle,
\intertext{%
into admissible subcategories and semi-orthogonal decompositions
}
\D^\locbd_\pc(P) & = \langle \im \Phi^\pc_a, \ldots, \Phi^\pc_{a + n}\rangle,\\
\D_\sg(P) & = \langle \im \Phi^\sg_a, \ldots, \Phi^\sg_{a + n}\rangle
\end{align}
into right admissible subcategories.
In particular, if $S$ is Noetherian, we get the corresponding decomposition of $\D^\bd_\co(P) = \D^\locbd_\pc(P)$.
\end{corollary}

\begin{proof}
The functors $\Phi_i$ and their right adjoints restrict to the subcategories
of perfect complexes and the subcategories of locally bounded pseudo-coherent
complexes, respectively.
Indeed, this follows from the corresponding properties for $\Psi_i$ (cf.~\cite[Corollary~6.8]{BS-conservative})
using the notation in the proof of Theorem~\ref{t:sod-brauer-severi}.
Hence the decompositions in the corollary into right admissible categories
follow formally from Theorem~\ref{t:sod-brauer-severi} (see~\cite[Corollary~5.19]{BS-conservative}).
Moreover, the restriction of each $\Psi_i$ to the category of
perfect complexes has a left adjoint
(cf.~\cite[Corollary~6.8]{BS-conservative}).
\end{proof}

\begin{remark}
\label{r:bernardara}
Let us explain why Bernardara's Theorem~4.1 from~\cite{bernardara-BS-schemes} is a special case of the semi-orthogonal
decomposition~\eqref{eq-sod-perfect}.
In Bernardara's setting,
the base $S$ is a Noetherian, separated scheme with the additional property that any two points is
contained in an open affine.
In fact, he only states that $S$ should be locally noetherian with the extra condition on points,
but the other properties are used implicitly since he uses results from C{\u{a}}ld{\u{a}}raru's thesis,
where the stronger assumptions are in effect \cite[Section~2.1]{caldararu2000}.
Bernardara's category $\D(S, \alpha)$ is the derived category of perfect complexes of $\alpha$-twisted coherent sheaves.
By the comparison by Lieblich mentioned in Remark~\ref{rem-twisted-2-cocycle},
this is equivalent to $\D_\pf(\coh_\alpha(S))$,
which is easily seen to be equivalent to $\D_\pf(\Qcoh_\alpha(S))$ by using that every quasi-coherent
sheaf on the gerbe corresponding to $\alpha$ is a direct limit of its coherent subsheaves~\cite[Theorem~15.4]{lmb2000}.
Now this category is equivalent to $\D_{\pf, \alpha}(S)$ since the functor \eqref{eq-equiv-qcoh} is an equivalence
under the assumptions on $S$ (see~Remark~\ref{rem-trivial-case}).
\end{remark}


\bibliographystyle{myalpha}
\bibliography{references}
\end{document}